\documentclass[11pt]{amsart}

\usepackage{amsmath}
\usepackage{amssymb}
\usepackage{graphicx}
\usepackage{geometry}
\usepackage{enumerate}  
\newtheorem{theorem}{Theorem}[section]
\newtheorem{corollary}[theorem]{Corollary}
\newtheorem{lemma}[theorem]{Lemma}
\newtheorem{proposition}[theorem]{Proposition}
\theoremstyle{definition}
\newtheorem{definition}[theorem]{Definition}

\newtheorem{remark}[theorem]{Remark}
\numberwithin{equation}{section}

\title[Survey of Metric fixed point theory]{Survey of Metric fixed point theory in random functional analysis}

\author[T. X. Guo]{Tiexin Guo*}
\address[T. X. Guo]{School of Mathematics and Statistics, Central South University,
	Changsha {\rm 410083}, Hunan Province, P. R. China}
\email{tiexinguo@csu.edu.cn}
\thanks{*Corresponding author}

\author[Q. Tu]{Qiang Tu}
\address[Q. Tu]{School of Mathematics and Statistics, Lingnan Normal University,
	Zhanjiang {\rm524048}, Guangdong Province, P. R. China}
\email{qiangtu126@126.com}

\author[X. H. Mu]{Xiaohuan Mu}

\address[X. H. Mu]{School of Mathematics and Statistics, Lingnan Normal University,
	Zhanjiang {\rm524048}, Guangdong Province, P. R. China}
\email{xiaohuanmu@163.com}

\author[Y. Y. Sun]{Yuanyuan Sun}

\address[Y. Y. Sun]{School of Mathematics and Statistics, Central South University,
	Changsha {\rm 410083}, Hunan Province, P. R. China}
\email{{yuanyuansun1205@163.com}}



\begin{document}
\maketitle


\begin{center}
	\normalsize{ \textit{Dedicated to Professor Goong Chen on the occasion of his 75th birthday.}}
\end{center}

\begin{abstract}
Based on the idea of randomizing the traditional space theory of functional analysis, random functional analysis has been developed as functional analysis over random metric spaces, random normed modules and random locally convex modules. Since these random frameworks have much more complicated algebraic, topological and geometric structures than their prototypes, the development of fixed point theory in random functional analysis had been almost stagnant before 2010. Unexpectedly, with the deep development of stable set theory fixed point theory in random functional analysis, including both its metric and topological fixed point theory, has made considerable  progress in the recent 15 years. The purpose of this paper is to survey the important progress in metric fixed point theory in random functional analysis, including the random Banach contraction mapping principle and Caristi fixed point theorem on complete  random metric spaces,  and fixed point theorems for random nonexpansive and asymptotically nonexpansive mappings in complete random normed modules. Besides, the connections among the topics surveyed, random equations and random fixed point theorems for random operators are also briefly mentioned.
\end{abstract}

\maketitle
\makeatletter
\newcommand\blfootnote[1]{%
	\begingroup
	\renewcommand\thefootnote{}\footnote{#1}%
	\addtocounter{footnote}{-1}%
	\endgroup
}
\makeatother

\blfootnote{This work was supported by the National Natural Science Foundation of China (Grant Nos.12371141) and the Natural Science Foundation of Hunan Province of China (Grant No.2023JJ30642).}


\section{Introduction}
\label{intro}
\par
Fixed point theory in metric spaces mainly consists of metric fixed point theory and topological fixed point  theory. Metric fixed point theory includes the  celebrated Banach 
contraction mapping principle \cite{B1922},  Caristi fixed point theorem \cite{C1976}, Browder-G\"ohde-Kirk fixed point theorem for nonexpansive mappings \cite{B1965, Kirk1965, 
KS2001}, Brodskii-Milman-Lim common fixed point theorems for families of  isometric and nonexpansive   mappings \cite{BD1948, L1974a, L1974b, LLPV2003}, and Goebel-Kirk fixed point theorem \cite{KS2001} and Xu's demiclosedness principle \cite{X91} for asymptotically 
 nonexpansive mappings, see \cite{GR1984, Kirk1981, KS2001}  for the related  excellent  literatures surveying metric fixed point theory. Topological fixed point theory includes the celebrated 
 Brouwer and Schauder fixed point theorems \cite{B1912, S1930}.  Fixed point theory has rich applications in analysis,  topology, various kinds of equations, 
 optimization and mathematical economics \cite{GD2003}.

\par
Started in 1942, the idea of randomizing the traditional 
space theory or operator theory of functional analysis 
emerged. Because fixed point theory has the 
fundamental importance in functional analysis and its 
applications, fixed point theory has been naturally 
extended in the corresponding random ways. The following 
two paragraphs will give a brief historical review of 
the related developments

\par
The idea of randomizing the traditional space theory of 
functional analysis dates back to the theory of 
probabilistic metric spaces, which was initiated by K. 
Menger and subsequently founded by B. Schweizer, A. 
Sklar, A. N. \v{S}erstnev and others \cite{SS2005}. Perhaps, Menger's 
original idea was to place inherent uncertainties of the 
measurements in science into metric geometry, he 
introduced a probabilistic metric space as a 
probabilistic generalization of a metric space, in which 
the probabilistic distance between two points is defined 
as a distance distribution function, similarly, a 
probabilistic normed space was also introduced, see 
\cite{SS2005} for concrete details. Because K. Menger, B. 
Schweizer and A. Sklar emphasized the idea of defining 
the randomness of distances by distribution functions 
rather than random variables, the development of the 
theory of probabilistic metric spaces succeeded only in 
the study of the topological theory of probabilistic 
metric spaces and the development of fixed point 
theorems for contraction mappings with respect to 
probabilistic metric, whereas, compared with the deep 
theory of normed spaces, the theory of probabilistic 
normed spaces has not made any substantial progress yet 
so far, since most of probabilistic normed spaces are 
non-locally convex spaces under the conventionally used 
$(\varepsilon,\lambda)$-topology (that is to say, the theory of 
conjugate spaces, as a powerful tool for locally convex 
spaces, universally fails for probabilistic normed 
spaces). Therefore, although  random metric spaces 
and random normed spaces were earlier introduced in \cite[chapters 9 and 15]{SS2005}, they had not been deeply developed 
before 1983 since then they were only regarded as a subclass 
of probabilistic metric and normed spaces, respectively. 
Our breakthrough came in 1989 \cite{G1992, G1993}, by paying 
attention to the unique structure of random metric and 
normed spaces (for example, they possess the stronger 
triangle inequality  and richer measure-theoretic 
structure), proposing an equivalent formulation of their 
original definitions and in particular presenting the 
idea of random conjugate spaces suitable for the deep 
development of random normed spaces. The key step in our 
breakthrough is the introducing of random normed 
modules, which are an important class of random normed 
spaces and only for which the theory of random conjugate 
spaces can be so deeply developed that it may be 
comparable to the theory of conjugate spaces for normed 
spaces. Before 2009,  the theory of random normed 
modules and more general random locally convex modules 
has obtained  a systematic and deep development under the 
$(\varepsilon,\lambda)$-topology, see \cite{G2010} for details, 
in particular, some of these developments are consistent 
with some foundational studies on nonsmooth differential 
geometry on metric measure spaces developed by Gigli, Pasqualetto and others \cite{BPS2023,CLP2025,G2018, GMT2024, LPV2024}. In 2009, 
Filipovi\'{c} et al. \cite{FKV2009}  introduced the locally 
$L^{0}$-convex modules with the aim of establishing a 
random convex analysis over these modules to meet the needs of conditional 
convex risk measures, which also led them to introducing 
another kind of topology for random locally convex 
modules, called the locally $L^{0}$-convex topology. 
The  locally $L^{0}$-convex topology is much stronger 
than the $(\varepsilon,\lambda)$-topology, and even not 
a  linear topology. Following \cite{FKV2009}, Guo \cite{G2010} introduced 
the notion of $\sigma$-stability for subsets of an 
$L^{0}$-module and established the intrinsic connection 
between the basic theories derived from the two kinds of 
topologies for random locally convex modules. Since Guo's work \cite{G2010}, the theory of random normed modules and random 
locally convex modules has entered a new developing model, 
namely, a model of simultaneously considering the two 
kinds of topologies, which makes functional  analysis on 
random metric spaces, random normed modules and random 
locally convex modules, aptly called random functional 
analysis, have undergone a thorough development as a 
whole. In particular, stimulated by financial 
applications, random nonlinear functional analysis 
(mainly, fixed point theory in random normed modules) 
has made a remarkable progress in the past 15 years. In \cite{GWYZ2020, GY2012}, Ekeland variational principle and Caristi 
fixed point theorem on complete random metric 
spaces were established. Since it does not make sense to speak of weak compactness for a closed $L^{0}$-convex subset of a random normed module, in 2020 the notion of an 
$L^{0}$-convex compactness introduced in \cite{GZWW2021} as a proper substitution for the notion of weak compactness was used 
to establish the Browder-G\"{o}hde-Kirk fixed point 
theorem for nonexpansive mappings in complete random 
normed modules in \cite{GZWG2021}. In 2025, the stable set theory 
as an abstract generalization of $\sigma$-stability was 
employed in \cite{MTGX2025} in an unexpected way to introduce the 
notion of a complete random normal structure and prove 
that it is equivalent to the notion of random normal 
structure for an $L^{0}$-convexly compact set, so that 
we can establish several important common fixed point theorems for
 families of isometric and nonexpansive mappings in complete 
random normed modules, which solves a central problem in 
metric fixed point theory in random functional analysis. 
Recently, by means of the relation between geometry of 
random normed modules and geometry of normed spaces, several important
fixed point theorems for a random asymptotically 
nonexpansive mapping in a complete random uniformly 
convex random normed module were also given in \cite{SGT2025, SGT2026}. All the results, 
combined with the easier random contraction fixed point 
theorems in complete random metric spaces \cite{G1992,G1999,GWYZ2020}, 
have shaped an almost complete metric fixed point theory 
in random functional analysis. At the same time as we 
have made the considerable progress in the metric fixed 
point theory, we also have made important breakthroughs 
in the study of topological fixed point theory. Since 
most of closed $L^{0}$-convex subsets in random normed 
modules, which frequently occur in dynamic equilibrium 
theory, stochastic control and stochastic analysis are 
not compact in general, in \cite{GWXYC2025,TMGC2025} Guo and his coauthors 
developed a theory of random sequentially compact sets,  
proved that the closed $L^{0}$-convex subsets are all 
random sequentially compact, and further established the 
random Brouwer fixed point theorem in random Euclidean 
spaces \cite{TMGC2025} and random (also called noncompact) Schauder 
fixed point theorem in random normed modules \cite{GWXYC2025}, where 
the main challenge lies in overcoming noncompactness. 
Similarly, the other topological fixed point theorems \cite{TMG2024,TMGYS2025} in random functional analysis were also 
established. Since a thorough (or ripe) topological 
fixed point theory in random functional analysis will 
need a systematic theory of stable compactness, which is 
more general than the theory of random sequential 
compactness and being under development, this paper only 
gives a complete survey of metric fixed point theory in 
random functional analysis, we will leave the results \cite{GWXYC2025,TMG2024,TMGC2025,TMGYS2025} and nearly finished 
other topological fixed point theorems together to 
another future survey paper.

\par
Random operator theory (based on the idea of randomizing the traditional operator theory of functional analysis) was 
systematically initiated by the Prague school of 
probabilists (mainly, A. \v{S}pace\v{k} and O. Han\v{s}) 
in 1950s, with the aim of using random operator 
equations to precisely model various systems, see \cite{B1976,H1957} 
for details. In the following 20  years, random fixed 
point theorems for random operators had obtained a rapid
and deep development, Bharucha-Reid's survey  paper \cite{B1976}
is an important mark of such progress, where random versions 
of the Banach contraction mapping principle and Schauder 
fixed point theorem were summarized. Subsequently, 
fixed point theorems for nonexpansive 
random operators were also obtained, for example, see \cite{GZWY2020, Xu1993}
and the related references therein. Methods of 
random fixed point theory for random operators can be 
characterized as follows: the existence of 
random fixed points for a random operator depend on 
classical fixed point theorems, and then the proofs of 
measurability for them depend on either the method of 
successive approximations \cite{H1957,N1969} or measurable 
selection theorems \cite{W1977}.

\par
In concluding this introduction, we would like to 
briefly compare fixed point theory in random functional 
analysis and random point theory for random operators. 
The traditional space theory is the cornerstone of  
functional analysis. The development of randomizing the 
traditional space theory encountered many 
challenges from the outset, for example, random normed modules and 
random locally convex modules are random generalizations 
of normed spaces and locally convex spaces, 
respectively, but they are non-locally convex in general 
under the $(\varepsilon,\lambda)$-toplology, 
consequently, the theory of conjugate (or, dual) spaces that is a powerful tool for locally convex 
spaces, universally fails for these random frameworks, 
which has forced us to develop systematically the theory of 
random conjugate spaces. Nowadays, random normed modules 
have been used as the model spaces for conditional 
convex risk measures \cite{FKV2009,G2024}, which has stimulated a 
systematic development of random convex analysis \cite{G2024,GZWYYZ2017}. 
Thanks to the significant contributions of Gigli, Pasqualetto and others \cite{BPS2023,CLP2025,G2018,LPV2024}, the theory of random 
normed modules has played more and more essential roles 
in the development of nonsmooth differential geometry. 
The development of fixed point theory in random function 
analysis in the past 15 years is rather demanding, which 
needs the simultaneous employment of the two kinds of 
topologies, the theory of random conjugate spaces, the 
theory of stable compactness and geometry of random 
normed modules. In such a long and slow developing 
process, we have gradually found an interesting 
phenomenon: random fixed point theory for random 
operators can be regained as a direct corollary of fixed 
point theory in random functional analysis in a 
natural way, namely, by considering the random 
metric spaces or random normed modules consisting 
equivalence classes of random elements, lifting a random 
operator to an usual mapping defined on these random 
spaces and thus converting the random fixed point 
problems of  a random operator into those of the lifted 
operator, which is completely the method of functional 
analysis and even need not employ measurable selection 
theorems, see \cite{GWXYC2025,GWYZ2020,GZWY2020,MTGX2025} for details. Even more, when 
measurable selection theorems fail, our method can still 
deal with some complicated random fixed point problems, 
see \cite{MTGX2025} for interesting examples. Thus, we will also 
again mention the concrete details for such a phenomenon 
in the surveying process. We would like to point out that such a phenomenon itself not only displays the power of the idea of randomizing space theory but also reveals the essence of random fixed point theory for random operators.

\par
The remainder of this paper is organized as follows. 
Section \ref{section2} first provides some prerequisites for this 
paper and the reader's convenience, including some known 
basic concepts such as stable sets, random metric 
spaces, random normed modules and random operators 
together with some known propositions. Section \ref{section3} is 
devoted to surveying the random Banach contraction mapping 
principle on complete random metric spaces. Section \ref{section4} is 
devoted to surveying the Ekeland variational principle 
and Caristi fixed point theorem on complete random 
metric spaces. Section \ref{section5} is devoted to surveying the 
Browder-G\"ohde-Kirk fixed point theorem for 
nonexpansive mappings  in complete random normed modules. 
Section \ref{section6} is devoted to surveying the common fixed point 
theorems for families of nonexpansive or isometric 
mappings in a complete random normed module. 
Section \ref{section7} is devoted to surveying fixed point theorems 
for random asymptotically nonexpansive mappings in a 
complete random uniformly convex random normed module. Finally, Section \ref{section8} concludes this paper with some remarks concerning the future possible applications of fixed  point theory in random functional analysis.


\section{Prerequisites}\label{section2}

For clarity, the section is divided into three subsections. \\

\noindent2.1. Stable sets.\\

In this section, $B:=(B,\vee, \wedge, ', 0,1)$ always denotes a
complete Boolean algebra, namely, a complete 
complemented distributive lattice with the smallest 
element $0$ and the greatest element $1$. A nonempty subset 
$\{a_{i},i\in I\}$ of B is called a partition of unity in B if 
$\bigvee_{i\in I}a_{i}=1$ and $a_{i}\wedge a_{j}=0$ for 
any $i,j\in I$ with $i\neq j$.
\begin{definition}[{\cite{GMT2024}}]\label{definition2.1}
	Let $X$ be a nonempty set. An equivalence relation 
	$\sim$ on $X\times B$ (denote by $x|a$ the 
	equivalence class of $(x,a)\in X\times B$) is said 
	to be $B$-regular if the following conditions are 
	satisfied:
	\begin{enumerate}[(1)]
	
\item	$x|a=y|b$ implies $b=a$;
	\item  $x|a=y|a$ implies $x|b=y|b$ for any $b\leq a$;
	\item $\{a\in B: x|a=y|a\}$ has a greatest element 
	for any $x$ and $y$ in $X$;
	\item $x|1=y|1$ implies $x=y$.
	\end{enumerate}
	In addition, given a $B$-regular equivalence 
	relation $\sim$ on $X \times B$, a nonempty subset 
	$G$ of $X$ is said to be $B$-stable (finitely 
	$B$-stable) with respect to $\sim$ if, for each 
	nonempty subset $\{x_{i}, i \in I\}$ (accordingly, 
	for each finite nonempty subset $\{x_{i}, i=1\sim n\}$) in 
	$G$ and each partition $\{a_{i}, i \in I\}$ 
	(accordingly, each finite partition $\{a_{i}, 
	i=1\sim n\}$) of unity, there exists $x \in G$ such 
	that $x|a_{i}= x_{i}|a_{i}$ for each $i \in I$ 
	(accordingly, $x|a_{i}= x_{i}|a_{i}$ for any 
	$i=1\sim n$). 
\end{definition}
\par
\begin{remark}\label{remark 2.2}
By (3) and (4) in Definition \ref{definition2.1}, 
	$x$ in the definition of $B$-stability  must be 
	unique, $x$ is always  denoted by $\sum_{i\in I} 
	x_{i}|a_{i}$, called the concatenation of $\{x_{i}, 
	i\in I\}$ along $\{a_{i},i\in I\}$. In \cite{GMT2024}, we 
	proved that a $B$-regular relation on $X\times B$ 
	amounts to a $B$-valued metric $d$ on $X$, and that 
	$G$ is $B$-stable with respect to $\sim$ iff $G$ is 
	universally complete with respect to $d$, but our 
	notion of a $B$-stable set is more convenient for 
	applications in random functional analysis since it 
	is more  easier and direct to construct a 
	$B$-regular equivalence relation than a $B$-valued 
	metric, see \cite{KK1999} for the notions of a B-valued metric and a universally complete set. Finally, when $X$ itself is $B$-stable, our 
	notion of a $B$-stable set amounts to the notion of 
	a conditional set, which was introduced in \cite{DJKK2016}, but 
	our notion is still more convenient and direct in 
	applications since that of a conditional set was 
	presented in a non-traditional set theory and thus 
	is also inconvenient to apply. 
\end{remark}
\par
 Throughout this paper, let  $(\Omega,\mathcal{F},P)$ 
 always  denote a probability space, $\mathbb{K}$ the 
 scalar field $\mathbb{R}$ of real numbers or 
 $\mathbb{C}$ of complex numbers,  
 $L^{0}(\mathcal{F},\mathbb{K})$ the algebra of 
 equivalence classes of $\mathbb{K}$-valued 
 $\mathcal{F}$-measurable functions on $\Omega$, and 
 $B_{\mathcal{F}}$ the complete Boolean 
 algebra of equivalence classes of 
 $\mathcal{F}$-measurable subsets of $\Omega$. For any 
 $A \in \mathcal{F}$, we often use the corresponding 
 lower case $a$ for the equivalence class of $A$. 
 Conventionally, $B_{\mathcal{F}}$ is called the measure 
 algebra associated with $(\Omega, \mathcal{F}, P)$. It 
 is well known that $\{i \in I: a_{i} > 0\}$ must be at 
 most countable for any partition $\{a_{i}, i \in I\}$ 
 of unity in $B_{\mathcal{F}}$, and hence we may 
 consider only countable partitions of unity in 
 $B_{\mathcal{F}}$.
 \par
 We also always use $\mathbb{N}$ for the set of positive 
 integers. Let us recall that a left module $E$ over 
 $L^{0}(\mathcal{F}, \mathbb{K})$ (briefly, an 
 $L^{0}(\mathcal{F}, \mathbb{K})$-module) is said to be 
 regular if $E$ has the property: for any two elements 
 $x$ and $y$ in $E$, if there exists a countable 
 partition $\{A_{n}, n \in \mathbb{N}\}$ of $\Omega$ to 
 $\mathcal{F}$ such that $\tilde{I}_{A_{n}} x = 
 \tilde{I}_{A_{n}} y$ for each $n \in \mathbb{N}$, then 
 $x = y$, where $\tilde{I}_{A_{n}}$ denotes the 
 equivalence class of the characteristic function 
 $I_{A_{n}}$ of $A_{n}$ (namely, 
 $I_{A_{n}}(\omega) = 1$ when $\omega \in A_{n}$ and $0$ 
 otherwise).
 \par
 From now on, we always assume that all the 
 $L^0(\mathcal{F}, \mathbb{K})$-modules occurring in 
 this paper are regular since all random normed modules 
 are  known to be regular.
 
 \begin{definition}[{\cite{G2010}}]\label{definition2.3}
   Let $E$ be an $L^0(\mathcal{F}, \mathbb{K})$-module. 
   A nonempty subset $G$ of $E$ is said to be finitely 
   stable if $\tilde{I}_{A} x + \tilde{I}_{A^{c}} y \in 
   G$ for any $x$ and $y$ in $G$ and any $A \in 
   \mathcal{F}$, where $A^{c} = \Omega \setminus A$. A 
   nonempty subset $G$ of $E$ is said to be 
   $\sigma$-stable if, for each sequence $\{x_{n}, n \in 
   \mathbb{N}\}$ in $G$ and any countable partition 
   $\{A_{n}, n \in \mathbb{N}\}$ of $\Omega$ to 
   $\mathcal{F}$, there exists $x \in G$ such that 
   $\tilde{I}_{A_{n}} x = \tilde{I}_{A_{n}} x_{n}$ for 
   each $n \in \mathbb{N}$ (such an $x$ must be unique 
   since $E$ is regular), denoted by $\sum_{n \in 
   	\mathbb{N}} \tilde{I}_{A_{n}} x_{n}$,  and called 
   	the countable concatenation of $\{x_{n}, n \in 
   	\mathbb{N}\}$ along $\{A_{n}, n \in \mathbb{N}\}$.
\end{definition}
\par
\begin{remark}\label{remark2.4}
Let $E$ be an $L^0(\mathcal{F}, \mathbb{K})$-module. 
Define a relation $\sim$ on $E \times B_{\mathcal{F}}$ 
by $(x, a) \sim (y, b)$ if $a = b$ and $I_{a} x = I_{a} 
y$, where $I_{a} = \tilde{I}_{A}$ for $a = [A]$ (namely, $a$ is
the equivalence class of $A \in \mathcal{F}$). It is 
easy to see that $\sim$ is a $B_{\mathcal{F}}$-regular 
equivalence relation and that a nonempty subset $G$ of 
$E$ is $\sigma$-stable iff $G$ is 
$B_{\mathcal{F}}$-stable with respect to $\sim$. If $G$ 
is $\sigma$-stable, we can also easily verify that 
$\sum_{n \in \mathbb{N}} \tilde{I}_{A_{n}} x_{n} = 
\sum_{n \in \mathbb{N}} x_{n} | a_{n}$ for any sequence 
$\{x_{n}, n \in \mathbb{N}\}$ in $G$ and any countable 
partition $\{A_{n}, n \in \mathbb{N}\}$ of $\Omega$ to 
$\mathcal{F}$, where $a_{n}$ denotes the equivalence 
class of $A_{n}$ according to our convention. Therefore, 
the notion of a $B$-stable set is exactly an abstract 
generalization of that of a $\sigma$-stable set from an 
$L^0(\mathcal{F}, \mathbb{K})$-module to any nonempty set!
\end{remark}
\par
\begin{definition}\label{definition2.5}
Let $X$ and $Y$ be two $B$-stable sets. A mapping (or, function) 
$f: X \to Y$ is said to be $B$-stable if $f(\sum_{i \in 
	I} x_{i} | a_{i}) = \sum_{i \in I} f(x_{i}) | a_{i}$ 
	for any nonempty subset $\{x_{i}, i \in I\}$ in $X$ 
	and any partition $\{a_{i}, i \in I\}$ of unity.
\end{definition}

\par
When $X$ and $Y$ are two $\sigma$-stable subsets of two 
$L^{0}(\mathcal{F}, \mathbb{K})$-modules, respectively, a 
$\sigma$-stable mapping $f: X \to Y$, introduced in 
\cite{GZWG2021}, is exactly a special case of a 
$B_{\mathcal{F}}$-stable mapping in the sense of Definition \ref{definition2.5}.
\par
\begin{definition}\label{definition2.6}
Let $X$ be a $B$-stable set and $\mathcal{E}$ be a 
nonempty family of $B$-stable subsets of $X$. Further, 
let $s(\mathcal{E}) = \{\sum_{i \in I} G_{i} | a_{i}$: 
$I$ is a nonempty index set, $\{G_{i}, i \in I\}$  a
nonempty subfamily of $\mathcal{E}$, $\{a_{i}, i \in 
I\}$ a partition of unity in $B$\}, called the 
stabilized family of $\mathcal{E}$, where $\sum_{i \in I} G_{i} | a_{i} = 
\{\sum_{i \in I} g_{i} | a_{i}: g_{i} \in G_{i} \text{ 
	for any } i \in I\}$. $\mathcal{E}$ is said to be 
	$B$-stable if $s(\mathcal{E}) = \mathcal{E}$. In addition, let $\mathcal{E}=\{G_{\alpha}, \alpha \in 
	\Lambda\}$ be a nonempty family of $B$-stable subsets of 
	$X$ and $\Lambda$ be also a $B$-stable set, then 
	$\mathcal{E}$ is said to be a $B$-consistent family if 
	$G_{\sum_{i \in I} \alpha_{i} | a_{i}} = \sum_{i \in I} 
	G_{\alpha_i} | a_i$ for any nonempty subset $\{\alpha_i, 
	i \in I\}$ of $\Lambda$ and any partition $\{a_i, i \in 
	I\}$ of unity in $B$.
\end{definition}	
Clearly, a B-consistent family must be a $\sigma$-stable family. $A$ special case of a $B$-consistent family, where 
$\Lambda$ is a $B_{\mathcal{F}}$-stable set and $\{G_\alpha, 
\alpha \in \Lambda\}$ is a nonempty subfamily of 
$\sigma$-stable subsets of a $\sigma$-stable set of an 
$L^0(\mathcal{F}, \mathbb{K})$-module, had been 
introduced in \cite[Definition 1.6]{MTGX2025} and played a crucial 
role in \cite{MTGX2025}.\\

\noindent2.2. Random metric spaces and random normed modules.\\

Throughout this paper, let $\bar{L}^{0}(\mathcal{F})$ be 
the set of equivalence classes of extended real-valued random variables on 
$(\Omega, \mathcal{F}, P)$ and $L^{0}(\mathcal{F}):= 
L^{0}(\mathcal{F}, \mathbb{R})$. It is well known from \cite{DS1958}
that $\bar{L}^{0}(\mathcal{F})$ is a complete lattice under the partial order $\leqslant$: $\xi \leqslant \eta$ iff $\xi^0(\omega) \leqslant \eta^0(\omega)$ a.s., where $\xi^{0}$ and $\eta^{0}$ are arbitrarily chosen representatives of $\xi$ and $\eta$, respectively. $\bigvee H$ and $\bigwedge H$ stand for the supremum and infimum of a nonempty subset $H$ of $\bar{L}^{0}(\mathcal{F})$, respectively. Also, $L^0(\mathcal{F})$ is a Dedekind complete lattice as a sublattice of $\bar{L}^{0}(\mathcal{F})$.
\par
As usual, for any $\xi$ and $\eta$ in $\bar{L}^{0}(\mathcal{F})$, $\xi > \eta$ means $\xi \geq \eta$ but $\xi \neq \eta$, whereas $\xi > \eta$ on $A$ means $\xi^{0}(\omega) > \eta^{0}(\omega)$ for almost all $\omega \in A$, where $A \in \mathcal{F}$ and $\xi^{0}$ and $\eta^{0}$ are arbitrarily chosen representatives of $\xi$ and $\eta$, respectively.
\par
Besides, the following notations will also be employed:
\par $L^{0}_{+}(\mathcal{F})=\{\xi \in L^{0}(\mathcal{F})~|~\xi \geq 0\}$;
\par $L^{0}_{++}(\mathcal{F})=\{\xi \in L^{0}(\mathcal{F})~|~\xi>0~\text{on}~\Omega\}$.
\par

\begin{definition}[{\cite{G1992}}]\label{definition2.7}
An ordered pair $(E, d)$ is called a random metric space (briefly, an $RM$ space) with base $(\Omega, \mathcal{F}, P)$, where $E$ is a nonempty set and $d$ is a mapping from $E\times E$ to  $L^{0}_{+}(\mathcal{F})$ such that the following are satisfied:
\begin{enumerate}[(RM-1)]
	\item $d(x,y) = 0$ iff $x = y$;
	\item $d(x,y) = d(y,x)$ for any $x$ and $y$ in $E$;
	\item $d(x,z) \leqslant d(x,y) + d(y,z)$ for any $x,y,z$ in $E$.
\end{enumerate}
As usual, $d$ is called the random metric on $E$.
\end{definition}
\par
\begin{definition}[{\cite{G1992}}]\label{definition2.8}
An ordered pair $(E, \|\cdot\|)$ is called a random normed space (briefly, an $RN$ space) over $\mathbb{K}$ with base $(\Omega, \mathcal{F}, P)$, where $E$ is a linear space over $\mathbb{K}$ and $\|\cdot\|$ is a mapping from $E$ to $L^{0}_{+}(\mathcal{F})$ such that the following are satisfied:
\begin{enumerate}[(RN-1)]
	\item $\|x\| = 0$ iff $x = \theta$ (the null in $E$);
	\item $\|\alpha x\| = |\alpha| \|x\|$ for any $\alpha \in \mathbb{K}$ and any $x \in E$;
	\item $\|x + y\| \leqslant \|x\| + \|y\|$ for any $x$ and $y$ in $E$.
\end{enumerate}
As usual, $\|\cdot\|$ is called the random norm on $E$.
If, in addition, $E$ is an $L^0(\mathcal{F}, \mathbb{K})$-module such that the following
\begin{enumerate}[(RNM-1)]
	\item $\|\xi x\| = |\xi| \|x\|$ for any $\xi \in L^{0}(\mathcal{F}, \mathbb{K})$ and $x \in E$.
\end{enumerate}
is also satisfied, then the $RN$ space $(E, \|\cdot\|)$ is called a random normed module (briefly, an $RN$ module) over $\mathbb{K}$ with base $(\Omega, \mathcal{F}, P)$, in which case the random norm $\|\cdot\|$ is also called an $L^{0}$-norm.
\end{definition}
\par
Obviously, when $(E, \|\cdot\|)$ is an $RN$ module, (RN-2) may be omitted since (RNM-1) has implied (RN-2). Historically, the original definition of an $RM$ space was introduced in \cite[Chapter 9]{SS2005}, where $d(x,y)$ is defined as a random variable. Similarly, the original definition of an $RN$ space was also mentioned in \cite[Chapter 15]{SS2005}. In Definitions \ref{definition2.7} and \ref{definition2.8}, we define them with $d(x,y)$ or $\|x\|$ as equivalence classes of random variables, which has provided much convenience for applications in functional analysis. The introducing of an $RN$ module is a key step in the development of  random functional analysis since random functional analysis is merely an empty shell without the notion of an $RN$ module!
\par
\begin{definition}\label{definition2.9}
	 Let $(E, d)$ be an $RM$ space with base $(\Omega, \mathcal{F}, P)$. For any given positive numbers $\varepsilon$ and $\lambda$ with $0 < \lambda < 1$, let $U(\varepsilon, \lambda) = \{(x, y) \in E \times E: P\{\omega \in \Omega: d(x,y)(\omega) < \varepsilon\} > 1 - \lambda\}$, then $\{U(\varepsilon, \lambda): \varepsilon > 0, 0 < \lambda < 1\}$ forms a base for some metrizable uniform structure on $E$, called the $(\varepsilon, \lambda)$-uniformity induced by $d$ and denoted by $\mathcal{U}$. Further, if $(E, d)$ is $\mathcal{U}$-complete, then $(E, d)$ is said to be  $(\varepsilon, \lambda)$-complete. By the way, the topology induced by $\mathcal{U}$ is called the $(\varepsilon, \lambda)$-topology, denoted by $\mathcal{T}_{\varepsilon, \lambda}$.
\end{definition}
\par
The idea of introducing the $(\varepsilon, 
\lambda)$-uniformity and the $(\varepsilon, 
\lambda)$-topology is essentially due to B. Schweizer 
and A. Sklar \cite{SS2005}. For an $RN$ space $(E, \|\cdot\|)$, 
$E$ is, of course, an $RM$ space under the random  
metric $d$ defined by $d(x,y) = \|x - y\|$ for any $x$ 
and $y \in E$, it is easy to verify that 
$\mathcal{T}_{\varepsilon, \lambda}$ on $E$ is a linear topology, and thus we conventionally say that $(E, 
\|\cdot\|)$ is $\mathcal{T}_{\varepsilon, 
\lambda}$-complete if $(E, d)$ is $(\varepsilon,\lambda)$-complete. The algebra 
$L^0(\mathcal{F}, \mathbb{K})$ is clearly an $RN$ module 
under the $L^0$-norm $\|\cdot\|$ defined by $\|\xi\| = 
|\xi|$ for any $\xi \in L^0(\mathcal{F}, \mathbb{K})$, 
where $|\cdot|$ stands for the usual absolute value 
mapping, one can easily see that the $(\varepsilon, 
\lambda)$-topology on $L^0(\mathcal{F}, \mathbb{K})$ is 
exactly the topology of convergence in probability, and 
thus $L^0(\mathcal{F}, \mathbb{K})$ is a topological algebra over $\mathbb{K}$, namely the algebraic multiplication operation is jointly continuous. Conventionally, for any $RM$ space $(E, d)$, we always use $\mathcal{T}_{\varepsilon, \lambda}$ for the $(\varepsilon, \lambda)$-topology on $E$, then, for an $RN$ module $(E, \|\cdot\|)$ over $\mathbb{K}$ with base $(\Omega, \mathcal{F}, P)$, one can easily check that $(E, \mathcal{T}_{\varepsilon, \lambda})$ is a topological module over the topological algebra $(L^{0}(\mathcal{F}, \mathbb{K}), \mathcal{T}_{\varepsilon, \lambda})$.
\par
Motivated by the idea of Filipovi\'{c} et al.'s introducing the locally $L^{0}$-convex topology for a random normed module in \cite{FKV2009}, the following $L^{0}$-uniformity for an $RM$ space was introduced in \cite{GWYZ2020}.
\par

\begin{definition}[{\cite{GWYZ2020}}]\label{definition2.10}
 Let $(E, d)$ be an $RM$ space with base $(\Omega, \mathcal{F}, P)$. For any given $\varepsilon \in L^{0}_{++}(\mathcal{F})$, let $U(\varepsilon) = \{(x, y) \in E \times E: d(x,y) < \varepsilon \text{ on } \Omega\}$, then $\{U(\varepsilon): \varepsilon \in L^{0}_{++}(\mathcal{F})\}$ forms a base for some Hausdorff uniform structure on $E$, called the $L^{0}$-uniformity induced by $d$ and denoted by $\mathcal{U}_{c}$. Further, $(E, d)$ is said to be $L^{0}$-complete if $(E, \mathcal{U}_{c})$ is complete. The topology induced by $\mathcal{U}_{c}$ is called the $L^{0}$-topology for $E$, denoted by $\mathcal{T}_{c}$.
\end{definition}
\par
Just pointed out by Filipovi\'{c} et al. in \cite{FKV2009}, $(L^{0}(\mathcal{F}, \mathbb{K}), \mathcal{T}_{c})$ is merely a topological ring, namely $(L^{0}(\mathcal{F}, \mathbb{K}), \mathcal{T}_{c})$ is not a linear topological space in general, and $(E, \mathcal{T}_c{})$ is a topological module over $(L^{0}(\mathcal{F}, \mathbb{K}), \mathcal{T}_{c})$ when $(E, \|\cdot\|)$ is an $RN$ module over $\mathbb{K}$ with base $(\Omega, \mathcal{F}, P)$.
\par
\begin{definition}[{\cite{G1992}}]\label{definition2.11}
	Let $(E, \|\cdot\|)$ be an $RN$ space over $\mathbb{K}$ with base $(\Omega, \mathcal{F}, P)$. A linear operator $f: E \to L^{0}(\mathcal{F}, \mathbb{K})$ is called an almost surely (briefly, a.s.) bounded random linear functional on $E$ if there exists $\xi \in L^{0}_{+}(\mathcal{F})$ such that $|f(x)| \leq \xi \|x\|$ for any $x \in E$, $\|f\|:= \bigwedge \{\xi \in L^{0}_{+}(\mathcal{F}): |f(x)| \leq \xi \|x\| \text{ for any } x \in E\}$ is called the random norm of $f$. Denote by $E^{*}$ the linear space of a.s. bounded random linear functionals on $E$, it is clear that $E^{*}$ is also an $L^{0}(\mathcal{F}, \mathbb{K})$-module under the module multiplication defined by $(\xi \cdot f)(x) = \xi \cdot (f(x))$ for any $(\xi, f) \in L^{0}(\mathcal{F}, \mathbb{K}) \times E^{*}$ and any $x \in E$, and that $(E^{*}, \|\cdot\|)$ with $\|f\|$ defined as above for any $f \in E^{*}$, is also an $RN$ module over $\mathbb{K}$ with base $(\Omega, \mathcal{F}, P)$, called the random conjugate space of $(E, \|\cdot\|)$.
\end{definition}
\par

Let $(E, \|\cdot\|)$ be the same as in Definition \ref{definition2.11}. It is easy to see that an a.s. bounded linear operator $f: E \to L^{0}(\mathcal{F}, \mathbb{K})$ (namely, $f \in E^{*}$) must be continuous from $(E, \mathcal{T}_{\varepsilon, \lambda})$ to $(L^{0}(\mathcal{F}, \mathbb{K}), \mathcal{T}_{\varepsilon, \lambda})$, but a continuous 
linear operator from $(E, \mathcal{T}_{\varepsilon, \lambda})$ to $(L^0(\mathcal{F}, \mathbb{K}), \mathcal{T}_{\varepsilon, \lambda})$ is not necessarily a.s. bounded. However, 
when $(E, \|\cdot\|)$ is an $RN$ module, Guo proved in \cite{G1992} that a linear operator $f: E \to L^{0}({\mathcal{F},\mathbb{K}})$  is a.s. bounded iff $f$ is a continuous 
module homomorphisms from $(E, \mathcal{T}_{\varepsilon,\lambda})$ to $(L^{0}(\mathcal{F},\mathbb{K}), \mathcal{T}_{\varepsilon,\lambda})$, which is also the original motivation
for Guo to introduce an $RN$ module. Further, let $E_{c}^{*}$ be the $L^{0}(\mathcal{F}, \mathbb{K})$-module of continuous module homomorphisms from $(E, \mathcal{T}_{c})$ to 
$(L^{0}(\mathcal{F}, \mathbb{K}), \mathcal{T}_{c})$ for an $RN$ module $(E, \|\cdot\|)$ over $\mathbb{K}$ with base $(\Omega, \mathcal{F}, P)$, then Guo proved in \cite{G2010} that $E^{*} = E_{c}^{*}$, namely, $E$ has the same random conjugate space under the two kinds of topologies.
\par
In concluding the section, let us summarize the results concerning the relation between the $(\varepsilon, \lambda)$-completeness and the $L^{0}$-completeness for an $RM$ space, which were obtained in \cite{GWYZ2020} and will be used in Section \ref{section3} of this paper.
\par
\begin{definition}[{\cite{GWYZ2020}}]\label{definition2.12}
	Let $(E, d)$ be an $RM$ space with base $(\Omega, \mathcal{F}, P)$ and $G \subset E$ be a nonempty subset. $G$ is said to be finitely stable if, for any $x_{1}$ and $x_{2}$ 
in $G$ and any $A \in \mathcal{F}$, there exists $x \in G$ such that $\tilde{I}_{A} d(x, x_{1}) = 0$ and $\tilde{I}_{A^{c}} d(x, x_{2}) = 0$, $x$ is denoted by $x = 
\tilde{I}_{A} x_{1 }+ \tilde{I}_{A^{c}} x_{2}$ ($x$ must be unique by the triangle inequality). $G$ is said to be $\sigma$-stable if, for any sequence $\{x_{n}, n \in 
\mathbb{N}\}$ in $G$ and any countable partition $\{A_{n}, n \in \mathbb{N}\}$ of $\Omega$ to $\mathcal{F}$, there exists $x \in G$ such that $\tilde{I}_{A_{n}} d(x, x_{n}) = 0$ for each $n \in \mathbb{N}$, $x$ is denoted by $\sum_{n=1}^{\infty} \tilde{I}_{A_{n}} x_{n}$.\\
\end{definition}
\begin{remark}\label{remark2.13}
In the original terminology, the $\sigma$-stability introduced in \cite{GWYZ2020} is called the $d$-$\sigma$-stability since it appears that the $\sigma$-stability depends on 
the random metric $d$. However, we proved in \cite{GMT2024} that the $d$-$\sigma$-stability can be considered as a special case of the $\sigma$-stability introduced for a nonempty subset  in an $RN$ module, namely, Definition \ref{definition2.12} is essentially a special case of Definition \ref{definition2.3}.
\end{remark}
\par
\begin{remark}\label{remark2.14}
Let $(E, d)$ be an $RM$ space with base $(\Omega, \mathcal{F}, P)$. Define a relation $\sim$ on $E \times B_{\mathcal{F}}$ by $(x, a) \sim (y, b)$ if $a = b$ and $I_{a} d(x,y) = 0$ (where $I_{a} = \tilde{I}_{A}$ when $a$ is the equivalence class of $A \in \mathcal{F}$), then it is clear that $\sim$ is a $B_{\mathcal{F}}$-regular equivalence relation. Furthermore, for a nonempty subset $G$ of $E$, $G$ is $\sigma$-stable  iff $G$ is $B_{\mathcal{F}}$-stable with respect to $\sim$, in which case $\sum_{n=1}^{\infty} \tilde{I}_{A_{n}} x_{n}$ is exactly $\sum_{n=1}^{\infty} x_{n} | a_{n}$ for any sequence $\{x_{n}, n \in \mathbb{N}\}$ in $G$ and any countable partition $\{A_{n}, n \in \mathbb{N}\}$ of $\Omega$ to $\mathcal{F}$, where $a_{n}$ stands for the equivalence class of $A_{n}$.
\end{remark}
\par
Ramark \ref{remark2.4} and Remark \ref{remark2.14} together show that the notions of stability introduced in Definitions \ref{definition2.3} and \ref{definition2.12} can be unified to the whole Definition \ref{definition2.1}.
\par

Proposition \ref{proposition2.15} below was initially established in \cite{G2010} for the case of an $RN$ module and subsequently extended to an $RM$ space in \cite{GMT2024,GWYZ2020}.\\
\par
\begin{proposition}\label{proposition2.15}
Let $(E, d)$ be an $RM$ space with base $(\Omega, \mathcal{F}, P)$. Then we have the following statements:
\begin{enumerate}[{\rm(1)}]
	\item \rm{(\cite[Theorem 2.13]{GWYZ2020})}. If $E$ is $L^{0}$-complete, then $E$ is $(\varepsilon, \lambda)$-complete.
	\item (\cite[Theorem 2.5]{GWYZ2020}\rm{;} \cite[Theorem 2.18]{GMT2024}). If $E$ is finitely stable, then $E$ is $(\varepsilon, \lambda)$-complete iff both $E$ is $\sigma$-stable and $L^{0}$-complete.
\end{enumerate}
\end{proposition}
\par
\begin{definition}[{\cite{G1992}}]\label{definition2.16}
Let $(E, \|\cdot\|)$  be an $L^{0}(\mathcal{F}, \mathbb{K})$-module. A nonempty subset $G$ of $E$ is said to be $L^{0}$-convex if $\xi x + \eta y \in G$ 
for any $x$ and $y$ in $G$ and any $\xi$ and $\eta$ in $L^{0}_{+}(\mathcal{F})$ with $\xi + \eta = 1$.
\end{definition}
\par
Clearly, an $L^{0}$-convex set must be convex. Besides, an $L^{0}$-convex set and an $L^{0}(\mathcal{F}, \mathbb{K})$-module are both finitely stable. For the sake of convenience for the reader, we give the corollary below.
\par
\begin{corollary}[{\cite{G2010}}]\label{corollary2.17}
Let $(E, \|\cdot\|)$ be an $RN$ module and $G$ be an $L^{0}$-convex subset of $E$. Then $G$ is $\mathcal{T}_{\varepsilon, \lambda}$-complete iff both $G$ is $\sigma$-stable and 
$\mathcal{T}_{c}$-complete.
\end{corollary}
\par
\noindent2.3. Random elements and random operators.\\

\par
Random elements and random operators are two basic concepts in probabilistic functional analysis initiated by \v{S}pa\v{c}ek and Han\v{s} \cite{B1976,H1957}, in our writings we 
always employ the terminology ``random functional analysis'' to distinguish from the terminology ``probabilistic functional analysis''. The two terminologies have the essential 
difference: one is based on the idea of randomizing space theory, whereas the other is based on the idea of randomizing operator theory.
\par

\begin{definition}[{\cite{B1976,H1957}}]\label{definition2.18}
Let $(\Omega, \mathcal{F})$ be a measurable space and $(M, d)$ be a metric space. A mapping $V: \Omega \to M$ is called a random element if $V^{-1}(G) = \{\omega \in \Omega: 
V(\omega) \in G\} \in \mathcal{F}$ for each open subset $G$ of $M$. Further, a random element $V: \Omega \to M$ is said to be simple if $V$ only takes finitely many values.
\end{definition}
\par
Another convenient concept, called a strong (or, strongly measurable) random element , introduced by Lo\`{e}ve in the name of $M$-valued random variable (see \cite{G2010} or 
\cite{SS2005}) is as follows: a mapping $V: (\Omega, \mathcal{F}) \to (M, d)$ is called a strong random element if there exists a sequence $\{V_{n}, n \in \mathbb{N}\}$ of 
simple random elements from $\Omega$ to $M$ such that $\lim_{n \to \infty} d(V_{n}(\omega), V(\omega)) = 0$ for each $\omega \in \Omega$. It is well known that a random element 
becomes a strong random element iff $\{V(\omega): \omega \in \Omega\}$ is a separable subset of $M$, and thus the two notions coincide when $(M, d)$ is separable. The advantage 
of the notion of a strong random element lies in the fact that $d(V_{1}, V_{2})$ is a real-valued random variable when either of two random elements is strong. Also, $V_{1} + 
V_{2}$ is still strong for any two strong random elements with values in a normed space.
\par
\begin{definition}\rm{(\cite{B1976,H1957}\@)}.\label{definition2.19}
Let $(\Omega, \mathcal{F})$ be a measurable space, $(M_{1}, d_{1})$ and $(M_{2}, d_{2})$ be two metric spaces. A mapping $T: \Omega \times M_{1} \to M_{2}$ is called a random 
operator if $T(\cdot, x_{1}): \Omega \to M_{2}$ is a random element for any $x_{1} \in M_{1}$. If, in addition, for each $\omega \in \Omega$, $T(\omega, \cdot): M_{1} \to M_{2}$ is continuous, then $T$ is said to be a sample-continuous random operator.
\end{definition}

Similarly, a random operator $T: \Omega \times M_{1} \to M_{2}$ is said to be strong if $T(\cdot, x_{1})$ is a strong random element for any $x \in M_{1}$. It is easy to check that $T(\cdot, 
V(\cdot)): \Omega \to M_{2}$ is strong for any strong random element $V:\Omega\rightarrow M_{1}$ if $T: \Omega \times M_{1} \to M_{2}$ is a sample-continuous strong random 
operator.

Let us return to the convention that $(\Omega, \mathcal{F}, P)$ is always a probability space.
\par
\begin{definition}[{\cite{B1976,H1957}}]\label{definition2.20}
Let $(M, d)$ be a metric space and $T: (\Omega, \mathcal{F}, P) \times M \to M$ be a random operator. A random element $V: (\Omega, \mathcal{F}, P) \to M$ is called a random 
fixed point of $T$ if there exists a $\Omega_{0} \in \mathcal{F}$ with $P(\Omega_{0}) = 1$ such that $T(\omega, V(\omega)) = V(\omega)$ for each $\omega \in \Omega_{0}$.
\end{definition}
\par

Let $V_{1}$ and $V_{2}$ be two strong random elements from $(\Omega, \mathcal{F}, P)$ to a metric space $(M, d)$. $V_{1}$ and $V_{2}$ are said to be equivalent if $P\{\omega \in \Omega: d(V_{1}(\omega), V_{2}(\omega))\\
 = 0\} = 1$. Further, let $L^{0}(\mathcal{F}, M)$ be the set of equivalence classes of strong random elements from $(\Omega, \mathcal{F}, P)$ to $(M, d)$, define $d: 
 L^{0}(\mathcal{F}, M) \times L^{0}(\mathcal{F}, M) \to L^{0}_+(\mathcal{F})$ by $d(x, y)= $ the equivalence class of $d(x^{0}(\cdot), y^{0}(\cdot))$, where $x^{0}$ and $y^{0}$ 
 are arbitrarily chosen representatives of $x$ and $y$, respectively. Then $(L^{0}(\mathcal{F}, M), d)$ is an $RM$ space with base $(\Omega, \mathcal{F}, P)$. Moreover, $(L^{0}(\mathcal{F}, 
 M), d)$ is $\sigma$-stable: in fact, let $\{x_{n}, n \in \mathbb{N}\}$ be any sequence in $L^{0}(\mathcal{F}, M)$ and $\{A_{n}, n \in \mathbb{N}\}$ be any countable partition 
 of $\Omega$ to $\mathcal{F}$, further let $x_{n}^{0}$ be a given representative of $x_{n}$, and define $x^{0}: \Omega \to M$ by $x^{0}(\omega) = x_{n}^{0}(\omega)$ for $\omega \in 
 A_{n}$ for some $n \in \mathbb{N}$, then the equivalence class $x$ of $x^{0}$ satisfies $\tilde{I}_{A_{n}} d(x, x_{n}) = 0$ for any $n \in \mathbb{N}$, namely $x = 
 \sum_{n=1}^{\infty} \tilde{I}_{A_{n}} x_{n}$.
\par
Proposition \ref{proposition2.21} below is very simple but extremely useful.\\
\par
\begin{proposition}\label{proposition2.21}
Let $(M, d)$ be a metric space and $T: \Omega \times M \to M$ be a sample-continuous strong random operator. Further, define $\hat{T}: L^{0}(\mathcal{F}, M) \to L^{0}(\mathcal{F}, M)$ by $\hat{T}(x)=$ the equivalence class of~$T(\cdot, x^{0}(\cdot))$ for any $x \in L^{0}(\mathcal{F}, M)$, where $x^{0}$ is any chosen representative of $x$. Then we have the following statements:
\begin{enumerate}[{\rm(1)}]
	\item $\hat{T}$ is a $\sigma$-stable and $\mathcal{T}_{\varepsilon, \lambda}$-continuous operator.
	\item $x \in L^{0}(\mathcal{F}, M)$ is a fixed point of $\hat{T}$ iff $x^{0}$ is a random fixed point of $T$, where $x^{0}$ is any chosen representative of $x$.
\end{enumerate}
\end{proposition}
\par
$\hat{T}$ in Proposition \ref{proposition2.21} is called the lifted operator of $T$.
\par

\section{Random contraction fixed point theorems on complete $RM$ spaces}\label{section3}

\begin{theorem}[{\cite{G1992}}]\label{theorem3.1}
Let $(E, d)$ be an $(\varepsilon, \lambda)$-complete random metric space with base $(\Omega, \mathcal{F}, P)$, $\alpha \in L^{0}_{+}(\mathcal{F})$ with $\alpha < 1$ on $\Omega$ and $T: E \to E$ satisfy $d(T(x), T(y)) \leq \alpha d(x, y)$ for any $x$ and $y$ in $E$. Then we have the following:
\begin{enumerate}[{\rm(1)}]
	\item $T$ has a unique fixed point $x^{*}$ in $E$.
	\item $d(T^{n}(x), x^{*}) \leq \frac{\alpha^{n}}{1 - \alpha} d(x, T(x))$ for any $x \in E$ and $n \in \mathbb{N}$, where $T^{n}$ stands for the $n$-th iterates of $T$.
	\item$\{d(T^{n}(x), x^{*}), n \in \mathbb{N}\}$ converges a.s. to $0$ for any $x \in E$.
\end{enumerate}
\end{theorem}
\par
As usual, when $T$ is in Theorem \ref{theorem3.1} is relaxed to the condition that there exists some $n \in \mathbb{N}$ such that $d(T^{n}(x), T^{n}(y)) \leq \alpha d(x, y)$ for any $x$ and $y$ in $E$, $T$ still has a unique fixed point.
\par
In the sequel of this section, we will mostly consider the fixed point problems for a $\sigma$-stable mapping. Proposition \ref{proposition3.2} below can be proved in a parallel way with Lemma 2.11 of \cite{GZWG2021}, which shows that $\sigma$-stable mappings are rich enough.\\
\par
\begin{proposition}\label{proposition3.2}
Let $(E_{1}, d_{1})$ and $(E_{2}, d_{2})$ be two $\sigma$-stable $RM$ spaces with base $(\Omega, \mathcal{F}, P)$ and $T: E_{1} \to E_{2}$ be an $L^{0}$-Lipschitzian mapping \rm{(}namely, there exists $\xi \in L^{0}_{+}(\mathcal{F})$ such that $d_{2}(T(x), T(y)) \leq \xi d_{1}(x, y)$ for any $x$ and $y$ in $E_{1}$). Then $T$ is $\sigma$-stable.
\end{proposition}

\par

Denote by $L^{0}(\mathcal{F}, \mathbb{N})$ the set of equivalence classes of random variables from $(\Omega, \mathcal{F}, P)$ to $\mathbb{N}$. Let $(E, d)$ be a $\sigma$-stable $RM$ space with base $(\Omega, \mathcal{F}, P)$, $T: E \to E$ be a self-mapping. For any $n \in L^{0}(\mathcal{F}, \mathbb{N})$, $T^{n}: E \to E$ is defined by $T^{n}(x) = \sum_{k=1}^{\infty} \tilde{I}_{A_{k}} T^{k}(x)$, where $n$ is expressed as $n = \sum_{k=1}^{\infty} \tilde{I}_{A_{k}} k$. Clearly, when $T$ is also $\sigma$-stable, $T$ and $T^{n}$ is commutative.
\par
Theorem \ref{theorem3.3} below is very interesting and useful.
\par
\begin{theorem}[{\cite{GZWG2021}}]\label{theorem3.3}
Let $(E, d)$ be a $\sigma$-stable $(\varepsilon, \lambda)$-complete $RM$ space with base $(\Omega, \mathcal{F}, P)$ and $T: E\rightarrow E$ be a $\sigma$-stable mapping. If there exists some $n\in L^{0}(\mathcal{F},\mathbb{N})$ and $\alpha \in L^{0}_{+}(\mathcal{F})$ with $\alpha < 1$ on $\Omega$ such that $d(T^{n}(x), T^{n}(y)) \leq \alpha d(x, y)$ for any $x$ and $y$ in $E$, then $T$ has a unique fixed point.
\end{theorem}
\par
Theorem \ref{theorem3.3} had been used in \cite{GZWG2021} to study the existence and uniqueness for a class of backward stochastic differential equations, see \cite[Lemma 4.3]{GZWG2021} for details.
\par
Corollary \ref{corollary3.4} below, due to Han\v{s} \cite{H1957}, appeared in a complicated way, but we proved in \cite[Corollary 3.19]{GWYZ2020} that it is also a special case of Theorem \ref{theorem3.3}, since the lifted operator $\hat{T}$ of $T$ in Corollary \ref{corollary3.4} satisfies Theorem \ref{theorem3.3}. 
\par
\begin{corollary}[{\cite{H1957}}]\label{corollary3.4}
Let $(X, d)$ be a polish space and $T: \Omega \times X \to X$ be a sample-continuous random operator such that the following condition is satisfied:
$$
P( \bigcup_{m=1}^{\infty} \bigcup_{n=1}^{\infty} \bigcap_{x \in X} \bigcap_{y \in X}\{ \omega \in \Omega: d(T^{n}(\omega, x), T^{n}(\omega, y)) \leq (1 - \frac{1}{m}) d(x, y) \} ) = 1.
$$
Then $T$ has a random fixed point which is unique in the sense of equivalence.
\end{corollary}
\par
To generalize Nadler's multivalued contraction mapping principle \cite{N1969} to the random setting, we first introduce the following:
\par
\begin{definition}[{\cite{GWYZ2020}}]\label{definition3.5}
	 Let $(E, d)$ be an $RM$ space with base $(\Omega,\mathcal{F},P)$.
\begin{enumerate}[{\rm(1)}]
	\item A nonempty subset $G$ of $E$ is said to be  a.s. bounded if $D(G):= \bigvee \{d(x, y): x, y \in G\} \in L^{0}_{+}(\mathcal{F})$.
	\item Denote by $CB_{\sigma}(E)$ the family of nonempty, a.s. bounded, $\sigma$-stable and $\mathcal{T}_{\varepsilon, \lambda}$-closed subsets of $E$, define $H: CB_{\sigma}(E) \times CB_{\sigma}(E) \to L^{0}_{+}(\mathcal{F})$ as follows:
	$$
	H(G_1, G_2) = \max\{ \bigvee_{x_1 \in G_1} d(x_1, G_2), \bigvee_{x_2 \in G_2} d(x_2, G_1)\},
	$$
	where $d(x, G) = \bigwedge \{d(x, y): y \in G\}$ for any $x \in E$ and any nonempty subset $G$ of $E$.
	\end{enumerate}
	\end{definition}
	$H$, as usual, is called the random Hausdorff metric on $CB_{\sigma}(E)$. $(CB_{\sigma}(E), H)$ is still an $RM$ space with base $(\Omega, \mathcal{F}, P)$, and is also $(\varepsilon, \lambda)$-complete when $(E, d)$ is $(\varepsilon, \lambda)$-complete.
	\par
	\begin{theorem}[{\cite{GWYZ2020}}]\label{theorem3.6}
	Let $(E, d)$ be an $(\varepsilon, \lambda)$-complete $RM$ space with base $(\Omega, \mathcal{F}, P)$, $\alpha \in L^{0}_{+}(\mathcal{F})$ with $\alpha < 1$ on $\Omega$ and $T: E \to CB_{\sigma}(E)$ such that $H(T(x), T(y)) \leqslant \alpha d(x, y)$ for any $x$ and $y$ in $E$. Then there exists $x \in E$ such that $x \in T(x)$.
	\end{theorem}
	\par
	We also proved in \cite[Corollary 3.13]{GWYZ2020} that Iton's random multivalued contraction fixed point theorem \cite{I1977} can be regarded as a special case of Theorem \ref{theorem3.6}.
	\par
	In concluding this section, let us give the $L^{0}$-version of Banach contraction mapping principle without  proof since its proof is routine but somewhat tedious.\\
		\par
\par
\begin{theorem}\label{theorem3.7}
Let $(E, d)$ be a $\sigma$-stable and $L^{0}$-complete $RM$ space with base $(\Omega, \mathcal{F}, P)$, $\alpha \in L^{0}_{+}(\mathcal{F})$ with $\alpha < 1$ on $\Omega$ and $T: E \to E$ such that $d(T(x), T(y)) \leq \alpha d(x, y)$ for any $x$ and $y$ in $E$. Then we have the following:
	\begin{enumerate}[{\rm(1)}]
		\item $T$ has a unique fixed point $x^{*}$.
		\item $d(T^{n}(x), x{}^*) \leq \frac{\alpha^{n}}{1 - \alpha} d(x, T(x))$ for any $x \in E$ and $n \in L^{0}(\mathcal{F}, \mathbb{N})$, where $\alpha^{n}=\sum^{\infty}_{k=1}\tilde{I}_{A_{k}}\alpha^{k}$ for $n=\sum^{\infty}_{k=1}\tilde{I}_{A_{k}}\cdot k$.
		\item[(3)] $\{T^{n}(x), n \in L^{0}(\mathcal{F}, \mathbb{N})\}$ converges in $\mathcal{T}_{c}$ to $x^{*}$ for any $x \in E$.
	\end{enumerate}
\end{theorem}
\par
\begin{remark}\label{remark3.8}
By Proposition \ref{proposition3.2}, $T$ in Theorem \ref{theorem3.7} is $\sigma$-stable, and thus $T^{n+m} = T^{n} \circ T^{m}$ for any $n$ and $m$ in $L^{0}(\mathcal{F}, \mathbb{N})$. Although $\{T^{n}(x), n \in \mathbb{N}\}$ converges a.s. to $x^{*}$ by (3) of Theorem \ref{theorem3.1}, we can not prove that this sequence also converges in $\mathcal{T}_{c}$ to $x^{*}$, from which we can observe  that (2) and (3) are indeed interesting.
\end{remark}

\section{The Ekeland  variational principle and Caristi fixed point theorem on a complete random metric space}\label{section4}

The classical Ekeland variational principle \cite{E1974} and Caristi fixed point theorem \cite{C1976} are the powerful tools in nonlinear functional analysis. Motivated by random convex analysis and its applications to conditional convex risk measures \cite{G2024, GZWYYZ2017}, in \cite{GWYZ2020, GY2012} we generalized them to complete $RM$ spaces.
\par
\begin{definition}[{\cite{GWYZ2020, GY2012,GZWYYZ2017}}]\label{definition4.1}
Let $(E,d)$ be an $RM$ space with base $(\Omega,\mathcal{F},P)$ and $f$ be a function from $E$ to $\bar{L}^{0}(\mathcal{F})$. $f$ is said to be
\begin{enumerate}[{\rm(1)}]
\item proper if $f(x) > -\infty$ on $\Omega$ for any $x\in E$, and $dom(f):= \{x \in E: f(x) < +\infty \text{ on } \Omega\}$ is nonempty.
\item proper and $\sigma$-stable if $E$ is $\sigma$-stable and $f(\sum\limits_{n=1}^{\infty} \tilde{I}_{A_{n}} x_{n}) = \sum\limits_{n=1}^{\infty} \tilde{I}_{A_{n}} f(x_{n})$
for any sequence $\{x_{n}, n \in \mathbb{N}\}$ in $E$ and any countable partition $\{A_{n}, n \in \mathbb{N}\}$ of $\Omega$ to $\mathcal{F}$, where we make the convention that $0 \cdot (+\infty) = 0$.
\item proper and $\mathcal{T}_{\varepsilon,\lambda}$-lower semicontinuous if $epi(f) := \{(x,r) \in E \times L^{0}(\mathcal{F}): f(x) \leq r\}$ is closed in $(E,\mathcal{T}_{\varepsilon,\lambda}) \times (L^{0}(\mathcal{F}),\mathcal{T}_{\varepsilon,\lambda})$.
\item proper and $\mathcal{T}_{c}$-lower semicontinuous if $epi(f)$ is closed in $(E,\mathcal{T}_{c}) \times (L^{0}(\mathcal{F}),\mathcal{T}_{c})$.
\item proper and bounded from below (or lower bounded) if there exists $\xi \in L^{0}(\mathcal{F})$ such that $f(x) \geq \xi$ for any $x \in E$.
\end{enumerate}
\end{definition}
\par
\begin{remark}\label{remark4.2}
Let $(E,d)$ be an $RM$ space with base $(\Omega,\mathcal{F},P)$, then, as usual, $E \times L^{0}(\mathcal{F})$ is still an $RM$ space with base $(\Omega,\mathcal{F},P)$ under the random metric $d'$ defined by $d'((x_{1},r_{1}),(x_{2},r_{2})) = d(x_{1},x_{2}) + |r_{1} - r_{2}|$, and hence a proper $\mathcal{T}_{\varepsilon,\lambda}$-lower semicontinuous function $f: E \to \bar{L}^{0}(\mathcal{F})$ is also $\mathcal{T}_c$-lower semicontinuous. But, Lemma~\ref{lemma4.3} below shows that the two kinds of lower semicontinuities for a $\sigma$-stable function coincide in the case when $(E,d)$ is $\sigma$-stable!
\end{remark}
\par

\begin{lemma}[{\cite[Theorem~2.12]{GWYZ2020}}]\label{lemma4.3}
Let $(E,d)$ be an $RM$ space with base $(\Omega,\mathcal{F},P)$ and $G$ be a $\sigma$-stable subset of $E$. Then $\overline{G}_{\varepsilon,\lambda} = \overline{G}_{c}$, where $\overline{G}_{\varepsilon,\lambda}$ and $\overline{G}_{c}$ stand for the closures of $G$ under $\mathcal{T}_{\varepsilon,\lambda}$ and $\mathcal{T}_{c}$, respectively, and hence $G$ is $\mathcal{T}_{\varepsilon,\lambda}$-closed iff $G$ is $\mathcal{T}_{c}$-closed.
\end{lemma}
\par
Three corollaries below are interesting and useful.
\par
\begin{corollary}[{\cite[Theorem~3.5]{GWYZ2020}}]\label{corollary 4.4}
	Let $(E,d)$ be a $\sigma$-stable $RM$ space with base $(\Omega,\mathcal{F},P)$ and $f: E \to \bar{L}^{0}(\mathcal{F})$ be $\sigma$-stable and proper. Then $f$ is $\mathcal{T}_{\varepsilon,\lambda}$-lower semicontinuous iff $f$ is $\mathcal{T}_c$-lower semicontinuous (by noting that $epi(f)$ is a $\sigma$-stable subset in $E \times L^0(\mathcal{F})$).
\end{corollary}
\par
\begin{corollary}[{\cite{GZWG2021}}]\label{corollary4.5}
	Let $G$ be a $\sigma$-stable subset of $L^0(\mathcal{F})$ and $\varepsilon \in L_{++}^0(\mathcal{F})$. Then there exists some $g_{\varepsilon} \in G$ such that $\bigvee G < g_{\varepsilon} + \varepsilon$ on $\Omega$ ($\bigwedge G > g_{\varepsilon} - \varepsilon$ on $\Omega$) if $G$ is bounded from above (accordingly, bounded from below).
\end{corollary}
\par
\begin{corollary}[{\cite[Theorem~3.4]{GWYZ2020}}]\label{corollary4.6}
	Let $(E,d)$ be a $\sigma$-stable $RM$ space with base $(\Omega,\mathcal{F},P)$ and $f: E \to \bar{L}^{0}(\mathcal{F})$ be proper and $\sigma$-stable. If $f$ is bounded from below, then, for any given $\varepsilon \in L_{++}^{0}(\mathcal{F})$, there exists some $x_{\varepsilon} \in E$ such that $f(x_{\varepsilon}) < \bigwedge \{f(x): x \in E\} + \varepsilon$ on $\Omega$. Similarly, if $f$ is bounded from above, then, for any given $\varepsilon \in L_{++}^{0}(\mathcal{F})$, there exists some $x_{\varepsilon} \in E$ such that $f(x_{\varepsilon}) > \bigvee \{f(x): x \in E\} - \varepsilon$ on $\Omega$.
\end{corollary}
\par

\begin{theorem}[{\cite[Theorem~2.12]{GY2012}}]\label{theorem4.7}
	Let $(E,d)$ be an $(\varepsilon,\lambda)$-complete $RM$ space with base $(\Omega,\mathcal{F},P)$ and $f: E \to \bar{L}^{0}(\mathcal{F})$ be a proper, $\mathcal{T}_{\varepsilon,\lambda}$-lower semicontinuous and lower bounded function. Then, for each $x_{0} \in dom(f)$, there exists $v \in dom(f)$ such that the following are satisfied:
	\begin{enumerate}[{\rm(1)}]
		\item $f(v) \leq f(x_{0}) - d(x_{0}, v)$.
		\item $f(x) \not\leq f(v) - d(x, v)$ for any $x \neq v$, namely, there exists $A_{x} \in \mathcal{F}$ with $P(A_{x}) > 0$ such that $f(x) > f(v) - d(x, v)$ on $A_{x}$.
	\end{enumerate}
\end{theorem}
\par
Theorem~\ref{theorem4.8} below is essentially equivalent to Theorem~\ref{theorem4.7}.
\par
\begin{theorem}[{\cite[Theorem~2.13]{GY2012}}]\label{theorem4.8}
	Let $(E,d)$ and $f$ be the same as in Theorem~\ref{theorem4.7}. Then there exists $v \in E$ such that $f(x) \not\leq f(v) - d(x, v)$ for any $x \neq v$.
\end{theorem}
\par
The following Caristi fixed point theorem on a complete $RM$ space is equivalent to Theorem~\ref{theorem4.8}.

\par
\begin{theorem}[{\cite[Theorem~2.14]{GY2012}}]\label{theorem4.9}
	Let $(E,d)$ and $f$ be the same as in Theorem~\ref{theorem4.8}. If $T: E \to E$ satisfies $f(T(u)) + d(T(u), u) \leq f(u)$ for any $u \in E$, then $T$ has a fixed point.
\end{theorem}

\par
Corollary~\ref{corollary4.10} below is the $\mathcal{T}_{c}$-version of Theorem \ref{theorem4.9}
\par
\begin{corollary}\label{corollary4.10}
	Let $(E,d)$ be a $\sigma$-stable and $L^{0}$-complete $RM$ space with base $(\Omega,\mathcal{F},P)$ and $f: E \to \bar{L}^{0}(\mathcal{F})$ be a $\sigma$-stable, proper, lower bounded and $\mathcal{T}_{c}$-lower semicontinuous function. If $T: E \to E$ satisfies $f(T(u)) + d(T(u), u) \leq f(u)$ for any $u \in E$, then $T$ has a fixed point.
\end{corollary}
\begin{proof} It is clear by Theorem~\ref{theorem4.9}, Corollary~\ref{corollary 4.4} and (1) of  Proposition~\ref{proposition2.15}.
\end{proof}
In a series of applications of Theorem~\ref{theorem4.7}, see \cite{GY2012,GZWYYZ2017} for details, we want to locate $x_{0}$ and strengthen the inequality in (2), which eventually leads us to the precise Ekeland variational principle on a complete $RM$ space by means of Corollary~\ref{corollary4.6}.

\begin{theorem}[{\cite[Theorem~3.6]{GWYZ2020}}]\label{theorem4.11}
	Let $(E,d)$ be a $\sigma$-stable, $(\varepsilon,\lambda)$-complete $RM$ space with base $(\Omega,\mathcal{F},P)$, $\varepsilon \in L_{++}^0(\mathcal{F})$ and $f: E \to \bar{L}^{0}(\mathcal{F})$ proper, $\sigma$-stable, $\mathcal{T}_{\varepsilon,\lambda}$-lower semicontinuous and bounded from below. Then, for any given $x_{0}$ satisfying $f(x_{0}) \leq \bigwedge\{f(x): x \in E\} + \varepsilon$ and any given $\alpha \in L_{++}^{0}(\mathcal{F})$, there exists $z \in E$ such that the following hold:
	\begin{enumerate}[{\rm(1)}]
		\item $f(z) \leq f(x_{0}) - \alpha d(z, x_{0})$.
		\item $d(z, x_{0}) \leq \alpha^{-1} \varepsilon$.
		\item $f(x) > f(z) - \alpha d(x, z)$ for $x \neq z$ (where ``$>$'' means ``$\geq$'' and ``$\neq$'').
	\end{enumerate}
\end{theorem}
\par
\begin{remark}\label{remark4.12}
	By replacing only the $(\varepsilon,\lambda)$-completeness of $E$ and the $\mathcal{T}_{\varepsilon,\lambda}$-lower semicontinuity of $f$ in Theorem~\ref{theorem4.11} with the $L^{0}$-completeness of $E$ and the $\mathcal{T}_{c}$-lower semicontinuity of $f$, then one can see that Theorem \ref{theorem4.11} still holds, which is exactly the $\mathcal{T}_{c}$-version of Theorem \ref{theorem4.11}.
\end{remark}
\section{The Browder--G\"ohde--Kirk fixed point theorem in complete $RN$ modules}\label{section5}


The classical Browder--G\"ohde--Kirk fixed point 
theorem in Banach spaces \cite{Kirk1965} involves two 
ingredients: one is the concept of normal structure, 
and the other is that of weak compactness for a 
closed convex subset in a Banach space. It is not 
very difficult to generalize normal structure to a 
complete $RN$ module, it is, however, another matter, to generalize the concept of weak compactness, since an $RN$ module is essentially non-locally convex under the 
$(\varepsilon,\lambda)$-topology (even it does not 
make sense to speak of weak compactness in the case 
of an $RN$ module). Fortunately, motivated by 
\v{Z}itkovi\'{c}'s work on convex compactness \cite{Z2010}, 
we have develop an elegant theory of $L^{0}$-convex 
compactness in \cite{GZWW2021} by making full use of the 
theory of random conjugate spaces. Besides, by means 
of  geometry of $RN$ modules established in \cite{GZ2010}
we can also give a thorough treatment of random 
normal structure in \cite{GZWG2021}, where we had finished 
a perfect extension of the classical 
Browder--G\"ohde--Kirk fixed point theorem to a 
$\mathcal{T}_{\varepsilon,\lambda}$-complete $RN$ 
module. Further, in \cite{GZWY2020} we characterized when 
$L^{0}(\mathcal{F},V)$ has random normal structure and 
$L^{0}$-convex compactness for a closed convex subset 
$V$ of a Banach space, which enables us to give a 
thorough treatment of the random fixed point problem for a random nonexpansive operator.
\par	
For any two elements $\xi$ and $\eta$ in $\bar{L}^{0}(\mathcal{F})$, let $\xi^{0}$ and $\eta^{0}$ be arbitrarily chosen representatives of $\xi$ and $\eta$, respectively, and $[\xi>\eta]$ denote the equivalence class of $\{\omega\in\Omega: \xi^{0}(\omega)>\eta^{0}(\omega)\}$. Besides, we also use $(\xi>\eta)$ for some representative of $[\xi>\eta]$. Although $(\xi>\eta)$ may depend on a particular choice of $\xi^0$ and $\eta^0$, all the propositions in which $(\xi>\eta)$ is employed are independent of the particular choice of $(\xi>\eta)$ and hence no confusion may occur!
\par		
\begin{definition}[{\cite[Definition 3.3]{GZWG2021}}]\label{definition5.1}
Let $(E,\|\cdot\|)$ be a $\mathcal{T}_{\varepsilon,\lambda}$-complete $RN$ module over $\mathbb{K}$ with base $(\Omega,\mathcal{F},P)$. A $\mathcal{T}_{\varepsilon,\lambda}$-closed $L^{0}$-convex subset $G$ of $E$ is said to have random normal structure if, for each a.s. bounded, $\mathcal{T}_{\varepsilon,\lambda}$-closed and $L^0$-convex subset $H$ of $G$ such that $D(H):= \bigvee\{\|x-y\|: x,y \in H\}$ (called the random diameter of $H$) $>0$, there exists a nondiametral point $h$ of $H$, namely, $\bigvee\{\|h-x\|: x \in H\} < D(H)$ on $(D(H) > 0)$.
\end{definition}
\par
To study random normal structure, let 
$(E,\|\cdot\|)$ be a 
$\mathcal{T}_{\varepsilon,\lambda}$-complete $RN$ 
module over $\mathbb{K}$ with base 
$(\Omega,\mathcal{F},P)$ and $\xi = \bigvee\{\|x\|: x 
\in E\}$. Since $E$ is an 
$L^{0}(\mathcal{F},\mathbb{K})$-module, for an 
arbitrarily chosen representative $\xi^{0}$ of $\xi$, 
one can easily see $P\{\omega \in \Omega: 
\xi^{0}(\omega) = 0 \text{ or } +\infty\} = 1$, 
further let $supp(E) = \{\omega \in \Omega: 
\xi^{0}(\omega) = +\infty\}$, called the 
support of $E$ (it is unique a.s.). If $P(supp(E)) = 
1$, then $E$ is said to have full support. 
Clearly, for each Banach space over $\mathbb{K}$, 
$L^{0}(\mathcal{F},B)$ always has full support. When we study geometry of $RN$ modules we 
always assume, without loss of generality, assume 
that a $\mathcal{T}_{\varepsilon,\lambda}$-complete 
$RN$ module in question has full support, since a
general case all can be converted to the special 
case, see \cite[Lemma~3.4]{GZWG2021} for details.
\par
Geometry of $RN$ modules was started by Guo and Zeng in \cite{GZ2010,GZ2012}, let us first recall the definition of random uniform convexity as follows, we employ the following notation:
	$$\small{\varepsilon_{\mathcal{F}}[0,2]=\{\varepsilon\in 
		L^{0}_{++}(\mathcal{F})~|~\mbox{there exists a positive 
			number}~\lambda~\mbox{such that}~\lambda \leq
		\varepsilon \leq 2\}.}$$
	$$\small{\delta_{\mathcal{F}}[0,1]=\{\delta\in 
		L^{0}_{++}(\mathcal{F})~|~\mbox{there exists a positive 
			number}~ \eta~\mbox{such that}~\eta \leq \delta \leq
		1\}.}$$
\par
For a $\mathcal{T}_{\varepsilon,\lambda}$-complete $RN$ module $(E,\|\cdot\|)$, $A_{x} = (\|x\| > 0)$ for any $x \in E$, and $B_{x,y} = A_{x} \cap A_{y} \cap A_{x-y}$ for any $x$ and $y$ in $E$.
\par	
\begin{definition}[{\cite{G2010, GZWG2021}}]\label{definition5.2}
Let $(E, \|\cdot\|)$ be a  $\mathcal{T}_{\varepsilon,\lambda}$-complete $RN$ module over $\mathbb{K}$ with base $(\Omega, \mathcal{F},P)$. $E$ is said to be random uniformly convex if for each $\varepsilon\in \varepsilon_{\mathcal{F}}[0,2]$ there exists  $\delta\in \delta_{\mathcal{F}}[0,1]$ such that $\|x-y\|\geq \varepsilon $ on $D$ always implies $\|x+y\|\leq 2(1-\delta)$ on $D$ for any $x$, $y$ $\in$ $U(1)$ and any $D\in \mathcal{F}$ with $D\subset B_{x, y}$ and $P(D)>0$, where $U(1)=\{z\in E~|~\|z\|\leq 1\}$, called the random closed unit ball of $E$.
\end{definition}
\par

Let $p \in [1,+\infty)$ and $(E,\|\cdot\|)$ be a $\mathcal{T}_{\varepsilon,\lambda}$-complete $RN$ module over $\mathbb{K}$ with base $(\Omega,\mathcal{F},P)$. Further, define $\|\cdot\|_{p}: E \rightarrow [0,+\infty]$ by $\|x\|_{p} = (\int_{\Omega}\|x\|^{p}dP)^{\frac{1}{p}}$ for any $x \in E$, $L^{p}(E) = \{ x \in E: \|x\|_{p} < +\infty \}$, then $(L^{p}(E),\|\cdot\|_{p})$ is a Banach space. In \cite{GMT2024} we  proved that all the $L^{p}$-normed $L^{\infty}(\mathcal{F})$-modules introduced in \cite{G2018} have the form $L^{p}(E)$ for some $\mathcal{T}_{\varepsilon,\lambda}$-complete $RN$ module $(E,\|\cdot\|)$.
\par	
One of the main results of \cite{GZ2010,GZ2012} can now be summarized as Proposition~\ref{proposition5.3} below.
\par			
\begin{proposition}[{\cite{GZ2010,GZ2012}}]\label{proposition5.3}
A $\mathcal{T}_{\varepsilon,\lambda}$-complete $RN$  module $(E,\|\cdot\|)$ with base $(\Omega,\mathcal{F}, P)$ is random uniformly convex iff $(L^{p}(E),\|\cdot\|_{p})$ is uniformly convex for some $p\in (1,+\infty)$. Specially, a Banach space $B$ is uniformly convex iff $L^{0}(\mathcal{F},B)$ is random uniformly convex.
\end{proposition}
\par			
\begin{proposition}[{\cite[Theorem~3.6]{GZWG2021}}]\label{proposition5.4}
Every $\mathcal{T}_{\varepsilon,\lambda}$-closed $L^0$-convex subset of a $\mathcal{T}_{\varepsilon,\lambda}$-complete random uniformly convex $RN$ module has random normal structure.
\end{proposition}
\par			
\begin{proposition}[{\cite[Theorems~2.16]{GZWY2020};\cite[Collary~3.7]{GZWG2021}}]\label{proposition5.5}
Let $V$ be a closed convex subset of a Banach space such that $V$ is a weakly compact set with normal structure. Then $L^{0}(\mathcal{F},V)$, as a $\mathcal{T}_{\varepsilon,\lambda}$-closed $L^0$-convex subset of $L^{0}(\mathcal{F},B)$, has random normal structure. Specially, $L^{0}(\mathcal{F},V)$ has random normal structure when $B$ is uniformly convex.
\end{proposition}
\par
						
\begin{definition}[{\cite{GZWW2021}}]\label{definition5.6}
A $\mathcal{T}_{\varepsilon,\lambda}$-closed $L^0$-convex subset $G$ of an $RN$ module $(E,\|\cdot\|)$ is said to be $L^{0}$-convexly compact if any family of $\mathcal{T}_{\varepsilon,\lambda}$-closed $L^0$-convex subsets of $G$ has a nonempty intersection whenever the family has the finite intersection property.
\end{definition}
\par				
Although we have known that a $\mathcal{T}_{\varepsilon,\lambda}$-closed $L^0$-convex subset is $L^{0}$-convexly compact iff it is also convexly compact in the sense of \cite{Z2010}, our formulation has an advantage that allows us to  characterize it by means of the theory of random conjugate spaces!
\par

\begin{theorem}[{\cite{GZWW2021}}]\label{theorem5.7}
	A $\mathcal{T}_{\varepsilon,\lambda}$-closed $L^{0}$-convex subset $G$ of a $\mathcal{T}_{\varepsilon,\lambda}$-complete $RN$ module $(E,\|\cdot\|)$ over $\mathbb{K}$ with base $(\Omega,\mathcal{F},P)$ is $L^{0}$-convexly compact iff, for each $f \in E^{*}$, there exists $g_{0} \in G$ such that $Re(f(g_{0})) = \max\{ Re(f(g)) : g \in G \}$, where $Re(f(g))$ stands for the real part of $f(g)$ for any given $g \in G$.
\end{theorem}
\par
Let $p \in(1,+\infty)$ and $(E,\|\cdot\|)$ be a $\mathcal{T}_{\varepsilon,\lambda}$-complete $RN$ module. It is well known from \cite{G2010} that $E$ is random reflexive iff $L^{p}(E)$ is reflexive. The following two corollaries are very useful.
\par
\begin{corollary}[{\cite{GZWW2021}}]\label{corollary5.8}
Every a.s. bounded $\mathcal{T}_{\varepsilon,\lambda}$-closed $L^{0}$-convex subset of a $\mathcal{T}_{\varepsilon,\lambda}$-complete $RN$ module $(E,\|\cdot\|)$ is $L^{0}$-convexly compact iff $E$ is random reflexive.
\end{corollary}
\par	
\begin{corollary}[{\cite{GZWY2020}}]\label{corollary5.9}
A closed convex subset $V$ of a Banach space $B$ is weakly compact iff $L^0(\mathcal{F},V)$, as a $\mathcal{T}_{\varepsilon,\lambda}$-closed $L^0$-convex subset of $L^0(\mathcal{F},B)$, is $L^0$-convex compactness.
\end{corollary}

Theorem \ref{theorem5.10} below is just the so called Browder--G\"ohde--Kirk fixed point theorem in complete $RN$ modules.

\par
\begin{theorem}[{\cite{GZWG2021}}]\label{theorem5.10}
Let $(E,\|\cdot\|)$ be a $\mathcal{T}_{\varepsilon,\lambda}$-complete $RN$ module over $\mathbb{K}$ with base $(\Omega,\mathcal{F},P)$ and $G\subset E$ be a $\mathcal{T}_{\varepsilon,\lambda}$-closed $L^0$-convexly compact subset with random normal structure. Then every nonexpansive mapping $T$ from $G$ to $G$ (namely, $\|T(x)-T(y)\| \leq \|x-y\|$ for any $x$ and $y$ in $G$) has a fixed point.
\end{theorem}
\par		
Theorem~\ref{theorem5.10} had been used in the study of backward stochastic differential equations in \cite{GZWG2021}, besides it was also used in \cite{GZWY2020} for the study of random fixed point problems for random nonexpansive operators.
\par
	
\begin{corollary}\rm{(\cite{GZWY2020}\@)}.\label{corollary5.11}
Let $V$ be a weakly compact convex subset with normal structure of a Banach space $(B,\|\cdot\|)$ over $\mathbb{K}$ and $T: \Omega \times V \to V$ be a nonexpansive strong random operator. Then $T$ has a random fixed point, which is itself a strong random element from $\Omega$ to $V$.
\end{corollary}
\par
\noindent\textbf{Remark 5.12.}\label{remark4.12}
By considering the lifted operator $\hat{T}: L^{0}(\mathcal{F},V) \to L^{0}(\mathcal{F},V)$, then $\hat{T}$ satisfies Theorem~\ref{theorem5.10} by using Corollary~\ref{corollary5.9} and Proposition~~\ref{proposition5.5}, and thus we do not require the superfluous assumptions that $(\Omega,\mathcal{F},P)$ is complete and $V$ is separable, as made in \cite{Xu1993}. Besides, we do without any complicated arguments concerning the weak topology on $V$ in order to employ measurable selection theorems, as done in \cite{Xu1993}.

\par

\section{Common fixed point theorems for families of nonexpansive  or isometric mappings in complete $RN$ modules}\label{section6}
	
It is relatively easy to formulate and characterize random normal structure when we studied fixed point theorems for a single nonexpansive mapping in a complete $RN$ module. Once we come to the study of common fixed point theorems for families of nonexpansive mappings in a complete $RN$ module, we are forced to face the problem of how to introduce a fruitful notion of a random complete normal structure, since the historical research into the connection between normal structure and complete normal structure is a demanding work, as shown in \cite{BK1967, L1974a,L1974b,LLPV2003}. Thanks to the recent progress in stable set theory \cite{GMT2024}, we eventually solved the related problems in \cite{MTGX2025}.
\par			
\begin{definition}[{\cite[Definition~1.10]{MTGX2025}}]\label{definition6.1}
Let $(E,\|\cdot\|)$ be an $RN$ module over $\mathbb{K}$ with base $(\Omega,\mathcal{F},P)$, $G$ a nonempty subset of $E$, and $H$ an a.s. bounded nonempty subset of $E$. Define $R(H,\cdot): E \to L^0_+(\mathcal{F})$ by
$$
R(H,x) = \bigvee\{\|x-y\|: y \in H\}, \forall x \in E.
$$
Then $R(H,G):= \bigwedge\{R(H,x): x \in G\}$ and $C(H,G):= \{x \in G: R(H,x) = R(H,G)\}$ are called the random Chebyshev radius and random Chebyshev center of $H$ with respect to $G$, respectively. Specially, $R(H,H)$ and $C(H,H)$ are called the random Chebyshev radius and center of $H$, respectively.
\end{definition}
\par				
When $\mathcal{F}=\{\Omega,\emptyset\}$, an $RN$ module with base $(\Omega,\mathcal{F},P)$ automatically reduces to a normed space, and hence the notions of random Chebyshev radius and center to their classical prototypes \cite{BD1948}. However, Definition~6.2 below is a nontrivial generalization of that of complete normal structure \cite{BK1967}.
\par	
\begin{definition}[{\cite[Definition~1.12]{MTGX2025}}]\label{definition6.2}
Let $(E,\|\cdot\|)$ be a $\mathcal{T}_{\varepsilon,\lambda}$-complete $RN$ module over $\mathbb{K}$ with base $(\Omega,\mathcal{F},P)$. An a.s. bounded $\mathcal{T}_{\varepsilon,\lambda}$-closed $L^0$-convex subset $G$ of $E$ is said to have random complete normal structure if every $\mathcal{T}_{\varepsilon,\lambda}$-closed $L^0$-convex subset $W$ of $G$, with $W$ containing more than one point, satisfies the following condition:
\begin{enumerate}[{\rm(*)}]
\item For every decreasing consistent net $\{W_\alpha, \alpha \in \Lambda\}$ of $\sigma$-stable subsets of $W$, if $R(W_\alpha, W) = R(W, W)$, $\forall \alpha \in 
\Lambda$, then $[\bigcup_{\alpha \in \Lambda} C(W_{\alpha}, W) ]^{-}_{\varepsilon,\lambda}\neq \emptyset$ and for any measurable set $B \subset (D(W) > 0)$ with $P(B) > 0$, there 
exists $y_B \in W$ such that $\tilde{I}_{B} y_{B} \not \in \tilde{I}_{B} [\bigcup_{\alpha \in \Lambda}\subset C(W_{\alpha}, W)]^{-}_{\varepsilon,\lambda}$.
\end{enumerate}
\end{definition}
\par
	
\begin{remark}\label{remark6.3.}
In Definition~\ref{definition6.2}, $\{W_{\alpha}, \alpha \in \Lambda\}$ is a consistent family, which implies that $\Lambda$ is a $B_\mathcal{F}$-stable set, at the same time $\{W_\alpha, \alpha \in \Lambda\}$ is a decreasing net, which further means that $\Lambda$ is a directed set and $W_{\alpha_{1}} \supset W_{\alpha_{2}}$ whenever $\alpha_{2} \geq \alpha_{1}$. Besides, the reader may consult the comments following \cite[Definition~1.12]{MTGX2025} for the nontriviality of Definition~\ref{definition6.2}.
\end{remark}
	
\par	
Theorem~\ref{theorem6.4} below is, without doubt, a crucial result of \cite{MTGX2025}, which is an important generalization of the main result of \cite{L1974a}.
\par	
\begin{theorem}[{\cite[Theorem~1.14]{MTGX2025}}]\label{theorem6.4}
Let $(E,\|\cdot\|)$ be a $\mathcal{T}_{\varepsilon,\lambda}$-complete $RN$ module and $G$ be an $L^{0}$-convexly compact subset of $E$. Then $G$ has random complete normal structure iff $G$ has random normal structure.
\end{theorem}
\par			
Theorem~~\ref{theorem6.4} eventually leads us to the following:
\par			
\begin{theorem}[{\cite[Theorem~1.15]{MTGX2025}}]\label{theorem6.5}
Let $(E,\|\cdot\|)$ be a $\mathcal{T}_{\varepsilon,\lambda}$-complete $RN$ module, $G \subset E$ an $L^{0}$-convexly compact subset with random normal structure and $\mathcal{T}$ a commutative family of nonexpansive mappings from $G$ to $G$. Then $\mathcal{T}$ has a common fixed point.
\end{theorem}
\par			
When $T$ in Theorem~\ref{theorem5.10} is isometric, we can require that $T$ have a fixed point in $C(G, G)$, which also gives a random generalization of \cite[Theorem~2]{LLPV2003}.
\par

\begin{theorem}[{\cite[Theorem~1.16]{MTGX2025}}]\label{theorem6.6}
	Let $(E,\|\cdot\|)$ be a $\mathcal{T}_{\varepsilon,\lambda}$-complete $RN$ module and $G \subset E$ be an $L^0$-convexly compact subset with random normal structure. Then every isometric mapping $T$ from $G$ to $G$ has a fixed point in $C(G,G)$.
\end{theorem}
\par
Similarly to Theorem~\ref{theorem6.6}, Theorem~\ref{theorem6.7} below is a surjectively isometric version of Theorem~\ref{theorem6.5}, where the commutativity of $\mathcal{T}$ may be omitted, and which is a random generalization of Brodskii and Milman's result in \cite{BD1948}.
\par	
\begin{theorem}[{\cite[Theorem~1.17]{MTGX2025}}]\label{theorem6.7}
Let $(E,\|\cdot\|)$ be a $\mathcal{T}_{\varepsilon,\lambda}$-complete $RN$ module, $G \subset E$ be an $L^{0}$-convexly compact subset with random normal structure, and $\mathcal{T}$ a family of surjective and isometric mappings from $G$ to $G$. Then $\mathcal{T}$ has a common fixed point in $C(G,G)$.
\end{theorem}

We would like to mention Corollary 3.7 of \cite{MTGX2025}, which shows that $\mathcal{T}$ still has a common fixed point in $C(G,G)$ if $\mathcal{T}$ in Theorem \ref{theorem6.7} is a commutative family of nonexpansive mappings from $G$ to $G$ satisfying $[T(G)]^{-}_{\varepsilon,\lambda}=G$ for each $T\in \mathcal{T}$. In fact,  Corollary 3.7 of \cite{MTGX2025} gave a random generalization of \cite[Theorem 1.9]{RV2015}.

\par	
Similarly to Corollary~\ref{corollary5.11}, by lifting a random operator we may obtain the following three random fixed point theorems for random nonexpansive or isometric operators as an easy application of Theorems~6.5--6.7, respectively.
\par	

\begin{corollary}[{\cite[Theorem~1.18]{MTGX2025}}]\label{corollary6.8}
Let $(B,\|\cdot\|)$ be a Banach space, $V$ a weakly compact convex subset with normal structure and $\mathcal{T}$ a commutative family of strong random nonexpansive operators from $\Omega \times V$ to $V$. Then there exists a strong random element $x^{0}(\cdot): \Omega \to V$ such that $P\{\omega \in \Omega: T(\omega, x^{0}(\omega)) = x^{0}(\omega)\} = 1$ for any $T \in \mathcal{T}$.
\end{corollary}
\par	
\begin{corollary}\rm{(}\cite[Theorem~1.19]{MTGX2025}).\label{corollary6.9}
Let $(B,\|\cdot\|)$ and $V$ be the same as in Corollary~\ref{corollary6.8}. If $T: \Omega\times V\rightarrow V$ is a strong random  isometric operator, then there exists a strong random element $x^{0}(\cdot): \Omega \to c(V,V)$ such that $T(\omega, x^{0}(\omega)) = x^{0}(\omega)$ for almost all $\omega\in\Omega$, where $c(V,V)$ is the Chebyshev center of $V$.
\end{corollary}
\par	
\begin{corollary}\rm{(}\cite[Theorem~1.20]{MTGX2025}).\label{corollary 6.10}
	Let $(B,\|\cdot\|)$ and $V$ be the same as in Corollary~\ref{corollary6.8}. If $\mathcal{T}$ is a family of strong random surjective isometric operators from $\Omega \times V$ to $V$, then there exists a strong random element $x^{0}(\cdot): \Omega \to c(V,V)$ such that $P\{\omega \in \Omega: T(\omega, x^0(\omega)) = x^0(\omega)\} = 1$ for any $T \in \mathcal{T}$.
\end{corollary}
\par	
Just pointed out in \cite{MTGX2025}, all the measurable selection theorems currently available in \cite{W1977} fail for Corollary \ref{corollary6.8} and Corollary \ref{corollary 6.10} when the families $\mathcal{T}$ are uncountable. It seems to us that the measurable selection theorems in \cite{W1977} also fail even for a single strong random operator $T$ as in Corollary \ref{corollary5.11} or Corollary \ref{corollary6.9} whenever $V$ is not separable.
\par			
\section{Fixed point theorems for a random asymptotically nonexpansive mapping in a complete random uniformly convex $RN$ module}\label{section7}
\par			
In 1972, Goebel and Kirk established their fixed 
point theorem for asymptotically nonexpansive 
mappings in uniformly convex Banach spaces, see 
\cite{KS2001, SGT2025} for details, 
where some refined techniques concerning 
uniformly convex Banach spaces were used. 
However, the theory of random uniformly convex complete $RN$ 
modules is not so ripe as that of uniformly 
convex Banach spaces, it is impossible for us to 
generalize the Goebel and Kirk's result to the 
random setting only by mimicking their 
techniques. Motivated by the work \cite{GWXYC2025}, in \cite{SGT2025}
we made full use of the connection between random 
uniform convexity and uniform convexity (namely, 
Proposition \ref{proposition5.3}) so that we can decompose a 
random asymptotically nonexpansive mapping in a 
random uniformly convex $RN$ module into a series 
of smaller asymptotically nonexpansive mappings 
in a uniformly convex Banach space. In such an interesting way, we achieved our goals in \cite{SGT2025}. In 1991, Xu \cite{X91} established an elegant demiclosedness  principle for an asymptotically nonexpansive mapping in a uniformly convex Banach space, where the deep theory of weak topologies was involved. However, the theory of random conjugate spaces is much complicated than the theory of classical dual spaces, the theory of random weak topologies for $RN$ modules are being developed. In \cite{SGT2026} we were again forced to make full use of the intrinsic connection between random conjugate spaces and classical conjugate spaces (namely, Proposition~\ref{proposition7.5} below) in order to generalize  Xu's result \cite{X91} to the random setting by a decomposing technique stated above.
\par							
\begin{definition}[{\cite[Definition 1.5]{SGT2025}}]
\label{definition 7.1}
Let $(E,\|\cdot\|)$ be an $RN$ module with base $(\Omega,\mathcal{F},P)$ and $G$ be a nonempty subset of $E$. A mapping $f: G \to G$ is said to be random asymptotically nonexpansive if there exists a sequence $\{\xi_{m},{m \in \mathbb{N}}\}$ in $L^{0}_+(\mathcal{F})$ with $\{\xi_{m}, m \in \mathbb{N}\}$ convergent a.s. to $1$, such that $\|f^{m}(x) - f^{m}(y)\| \leq \xi_{m} \|x - y\|$ for any $x$ and $y$ in $G$ and any $m \in \mathbb{N}$. A mapping $f: G \to G$ is called an eventually random asymptotically nonexpansive mapping if there exist some $l \in \mathbb{N}$ and a sequence $\{\xi_{m},m \geq l\}$ in $L^{0}_+(\mathcal{F})$ with $\{\xi_{m}, m \ge l\}$ convergent a.s. to $1$, such that $\|f^{m}(x) - f^{m}(y)\| \leq \xi_{m} \|x - y\|$ for any $x$ and $y$ in $G$ and any $m \geq l$.
\end{definition}
\par	
\begin{theorem}[{\cite[Theorem 1.6]{SGT2025}}]\label{theorem 7.2}
	Let $(E,\|\cdot\|)$ be a $\mathcal{T}_{\varepsilon,\lambda}$-complete random uniformly convex $RN$ module with base $(\Omega,\mathcal{F},P)$ and $G$ be an a.s. bounded $\mathcal{T}_{\varepsilon,\lambda}$-closed $L^{0}$-convex subset of $E$. Then every random asymptotically nonexpansive mapping $f: G \to G$ has a fixed point.
\end{theorem}
\par
\begin{theorem}[{\cite[Theorem 1.7]{SGT2025}}]\label{theorem 7.3}
	Under the same assumptions as in Theorem~\ref{theorem 7.2}, then the set of fixed points of $f$ is $\mathcal{T}_{\varepsilon,\lambda}$-closed and $L^{0}$-convex.
\end{theorem}
\par
\begin{theorem}[{\cite[Theorem 1.8 and Remark 2.9]{SGT2025}}]\label{theorem 7.4}
	Let $(E,\|\cdot\|)$ and $G$ be the same as in Theorem~\ref{theorem 7.2}. If $f: G \to G$ is an eventually random asymptotically nonexpansive mapping, then $f$ has a fixed point and the set of fixed points of $f$ is $\mathcal{T}_{\varepsilon,\lambda}$-closed and $L^{0}$-convex.
\end{theorem}

\par
\begin{proposition}[{\cite{G2010}}]\label{proposition7.5}
Let $(E,\|\cdot\|)$ be an $RN$ module over $\mathbb{K}$ with base $(\Omega,\mathcal{F},P)$, $p$ any given positive number with $1 \leq p < +\infty$ and $q \in (1,+\infty]$ the H\"older conjugate number of $p$. Then $(L^{q}(E^*),\|\cdot\|_{q})$ is isometrically isomorphic onto $(L^{p}(E),\|\cdot\|_{p})'$ under the canonical mapping $T$, where $(L^{p}(E),\|\cdot\|_{p})'$ stands for the conjugate space of $(L^{p}(E),\|\cdot\|_{p})$ and $T(g): L^{p}(E) \to \mathbb{K}$ is defined by $T(g)(f) = \int_{\Omega} gf \, dP$ for any $g \in L^{q}(E^*)$ and any $f \in L^{p}(E)$.
\end{proposition}
\par
Proposition~\ref{proposition7.5} not only unifies all the dual representation theorems of  Lebesgue--Bochner function spaces but also has played a crucial role in the historical development of the theory of random conjugate spaces for $RN$ modules, see \cite{G2010} for details.
\par
\setcounter{theorem}{5}
\begin{theorem}[{\cite[Theorem~2.10]{SGT2026}}]\label{theorem 7.6}
Let $(E,\|\cdot\|)$ be a $\mathcal{T}_{\varepsilon,\lambda}$-complete random uniformly convex $RN$ module over $\mathbb{K}$ with base $(\Omega,\mathcal{F},P)$, $G$ an a.s. bounded, $\mathcal{T}_{\varepsilon,\lambda}$-closed $L^0$-convex subset of $E$ and $f: G \to G$ a random asymptotically nonexpansive mapping. Then $I-f$ is random demiclosed at $\theta$ (the null in $E$), namely, for each sequence $\{x_{n}, n \in \mathbb{N}\}$ in $G$, if $\{x_{n}, n \in \mathbb{N}\}$ converges in $\sigma(E,E^*)$ to $x$ and $\{(I-f)(x_{n}), n \in \mathbb{N}\}$ converges in $\mathcal{T}_{\varepsilon,\lambda}$ to $\theta$, then $(I-f)(x) = \theta$. Here $I$ is the identity operator on $E$, that $\{x_{n}, n \in \mathbb{N}\}$ converges in $\sigma(E,E^{*})$ to $x$ means that $\{f(x_{n}), n \in \mathbb{N}\}$ converges in probability to $f(x)$ for any given $f \in E^{*}$.
\end{theorem}
\par
\section{Concluding remarks}\label{section8}

Although this paper only gives a survey of metric fixed point theory in random functional analysis, which, combined with the works \cite{GWXYC2025, SGT2025,SGT2026,TMG2024,TMGC2025,TMGYS2025}, has given the reader a complete picture of the progress made in fixed point theory in random functional analysis in the past 15 years. From the survey, we believe that the reader should have seen that these developments of fixed point theory have considerably benefited from the whole developments of random functional analysis in the past 30 plus years. In particular, random functional analysis has obtained a rapid development in connection with its applications in random equations, dynamic mathematical finance and nonsmooth differential geometry on metric measure spaces in the past 10 years. Therefore, we may hope that these developments of fixed point theory will find more applications in the closely related fields.
\par


\section*{Acknowledgements}
We owe thanks to many: To Professors Goong Chen, Hongkun Xu and George Xianzhi Yuan for their substantial collaborations with us on fixed point theory in random functional analysis; to Professor Kehe Zhu for his kind support for our work in random functional analysis; to Professors Nicola Gigli and Enrico Pasqualetto for their enlightening academic exchanges; and to the continuous financial support over the years from the National Natural Science Foundation of China.



\begin{thebibliography}{00}



\bibitem{B1922}
S. Banach,
{\it Sur les op\'{e}rations dans les ensembles et leurs applications aux \'{e}quations int\'{e}grale},
Fund. Math.
{\bf 3} (1922),
133--181.


\bibitem{BK1967}
L. P. Belluce and W. A. Kirk,
{\it Nonexpansive mappings and fixed points in Banach spaces},
Illinois J. Math.
{\bf 11} (1967),
474--479.


\bibitem{B1976}
A. T. Bharucha-Reid,
{\it Fixed point theorems in probabilistic analysis},
Bull. Amer. Math. Soc.
{\bf 82}(5) (1976),
641--657.




\bibitem{BD1948}
M. S. Brodskii and D. P. Milman,
{\it On the center of a convex set (in Russian)},
Dokl. Akad. Nauk SSSR (N.S.)
{\bf 59} (1948),
837--840.




\bibitem{B1912}
L. E. J. Brouwer,
{\it \"{U}ber Abbildung von Mannigfaltigkeiten},
Math. Ann.
{\bf 71}(4) (1912),
598.



\bibitem{B1965}
F. E. Browder,
{\it Nonexpansive nonlinear operators in a Banach space},
Proc. NatL. Acad. Sci. U.S.A.
{\bf 54}(4) (1965),
1041--1044. 




\bibitem{BPS2023}
E. Bru\'{e}, E. Pasqualetto and D. Semola,
{\it Rectifiability of the reduced boundary for sets of finite perimeter over RCD($K,N$) spaces},
J. Eur. Math. Soc.
{\bf 25}(2) (2023),
413--465.



\bibitem{CLP2025}
E. Caputo, M. Lu\v{c}i\'{c} and E. Pasqualetto,
{\it Parallel transport on non-collapsed RCD(K,N) spaces},
J. Reine Angew Math.
{\bf 819} (2025),
135--204.



\bibitem{C1976}
J. Caristi,
{\it Fixed point theorems for mappings satisfying inwardness conditions},
Trans. Amer. Math. Soc.
{\bf 215} (1976),
241--251.







\bibitem{DJKK2016}
S. Drapeau, A. Jamneshan, M. Karlicek and M. Kupper,
{\it The algebra of conditional sets and the concepts of conditional topology and compactness},
J. Math. Anal. Appl.
{\bf 437} (2016),
561--589.









\bibitem{DS1958}
N. Dunford and J. T. Schwartz, {\it Linear Operators (I): General Theory}, 
Interscience, New York 
 (1958).







\bibitem{E1974}
I. Ekeland, {\it On the variational principle}, J. Math. Anal. Appl. 
{\bf 47} (1974), 324--353.





\bibitem{FKV2009}
D. Filipovi\'{c}, M. Kupper and N. Vogelpoth, 
{\it Separation and duality in locally $L^{0}$-convex modules}, J. Funct. Anal.
 {\bf 256} (2009), 
 3996--4029.







\bibitem{G2018}
N. Gigli, {\it Nonsmooth differential geometry~---~an approach tailored for spaces with Ricci curvature bounded from below }, Mem. Amer. Math. Soc.
 {\bf 251}(1196) (2018).









\bibitem{GR1984}
K. Goebel and S. Reich,
 {\it Uniform Convexity, Hyperbolic Geometry, and Nonexpansive mappings},
  Monogr. Textbooks in Pure and Appl. Math.
   {\bf 83}, Marcel Dekker, Inc., New York (1984).








\bibitem{GD2003}
A. Granas and J. Dugundji, 
{\it Fixed Point Theory}, 
Springer-Verlag, New York (2003).


\bibitem{G1992}
T. X. Guo, {\it Random metric theory and its applications}, Ph.D. thesis, Xi'an Jiaotong University, China (1992).







\bibitem{G1993}
T. X. Guo, { \it A new approach to probabilistic functional analysis}, in: Proceedings of the first China Postdoctoral Academic Conference, pp. 1150--1154, The China National Defense and Industry Press, Beijing (1993).







\bibitem{G1999}
T. X. Guo, {\it Some basic theories of random normed linear spaces and random inner product spaces}, Acta Anal. Funct. Appl.
{\bf 1} (1999),
 160--184.




\bibitem{G2010}
T. X. Guo,
{\it Relations between some basic results derived from two kinds of topologies},
J. Funct. Anal.
{\bf 258}(9) (2010),
3024--3047.




\bibitem{G2024}
T. X. Guo,
{\it Optimization of conditional convex risk measures},
Proceedings of the 8th International Congress of Chinese Mathematicians (Beijing, 2019), Vol. 1, pp. 347--371, 
International Press of Boston, Somerville (2024).




\bibitem{GMT2024}
T. X. Guo, X. H. Mu and Q. Tu,
{\it The relations among the notions of various kinds of stability and their applications},
Banach J. Math. Anal.
{\bf 18} (2024),
42.




\bibitem{GWXYC2025}
T. X. Guo, Y. C. Wang, H. K. Xu, G. Yuan and G. Chen,
{\it A noncompact Schauder fixed point theorem in random normed modules and its applications},
Math. Ann.
{\bf 391} (2025),
3863--3911.




\bibitem{GWYZ2020}
T. X. Guo, Y. C. Wang, B. X. Yang and E. X. Zhang,
{\it On $d$-$\sigma$-stability in random metric spaces and its applications},
J. Nonlinear Convex Anal.
{\bf 21}(6) (2020),
1297--1316.




\bibitem{GY2012}
T. X. Guo and Y. J. Yang,
{\it Ekeland's variational principle for an $\bar{L}^{0}$-valued function on a complete random metric space},
J. Math. Anal. Appl.
{\bf 389} (2012),
1--14.




\bibitem{GZ2010}
T. X. Guo and X. L. Zeng,
{\it Random strict convexity and random uniform convexity in random normed modules},
Nonlinear Anal.
{\bf 73}(5) (2010),
1239--1263.




\bibitem{GZ2012}
T. X. Guo and X. L. Zeng,
{\it An $L^0(\mathcal{F}, R)$-valued function's intermediate value theorem and its applications to random uniform convexity},
Acta Math. Sin.
{\bf 28}(5) (2012),
909--924.




\bibitem{GZWG2021}
T. X. Guo, E. X. Zhang, Y. C. Wang and Z. C. Guo,
{\it Two fixed point theorems in complete random normed modules and their applications to backward stochastic equations},
J. Math. Anal. Appl.
{\bf 483}(2) (2020),
123644.




\bibitem{GZWW2021}
T. X. Guo, E. X. Zhang, Y. C. Wang and M. Z. Wu,
{\it $L^{0}$-convex compactness and its applications to random convex optimization and random variational inequalities},
Optimization,
{\bf 70}(5--6) (2021),
937--971.



\bibitem{GZWYYZ2017}
T. X. Guo, E. X. Zhang, Y. C. Wang, M. Z. Wu, B. X. Yang, G. Yuan and X. L. Zeng
{\it On random convex analysis},
J. Nonlinear Convex Anal.
{\bf 18}(11)  (2017),
1967--1996.









\bibitem{GZWY2020}
T. X. Guo, E. X. Zhang, Y. C. Wang and G. Yuan,
{\it $L^0$-convex compactness and random normal structure in $L^0(\mathcal{F}, B)$},
Acta Math. Sci.
{\bf 40}(2)  (2020),
457--469.



\bibitem{H1957}
O. Han\v{s},
{\it Random fixed point theorems},
Transactions of the first Prague conference on information theory, statistical decision functions, random processes, pp. 105--125, Publishing House of the Czechoslovak Academy of Sciences, Prague (1957).





\bibitem{I1977}
S. Itoh,
{\it A random fixed point theorem for a multivalued contraction mapping},
Pacific J. Math.
{\bf 68} (1977),
85--90.






\bibitem{Kirk1965}
W. A. Kirk,
{\it A fixed point theorem for mappings which do not increase distances},
Amer. Math. Monthly
{\bf 72} (1965),
1004--1006.




\bibitem{Kirk1981}
W. A. Kirk,
{\it Fixed point theory for nonexpansive mappings},
Fixed point theory (Sherbrooke, Que, 1980) pp. 484--505, Lecture Notes in Math., Vol. 886, Springer, Berlin (1981).




\bibitem{KS2001}
A. W. Kirk and B. Sims,
{\it Handbook of metric fixed point theory},
Kluwer Academic Publishers, Dordrecht (2001).




\bibitem{L1974a}
T. C. Lim,
{\it Characterizations of normal structure},
Proc. Amer. Math. Soc.
{\bf 43}(2) (1974),
313--319.




\bibitem{L1974b}
T. C. Lim,
{\it A fixed point theorem for families on nonexpansive mappings},
Pacific J. Math.
{\bf 53}(2) (1974),
487--493.




\bibitem{LLPV2003}
T. C. Lim, P. K. Lin, C. Petalas and T. Vidalis,
{\it Fixed points of isometries on weakly compact convex sets},
J. Math. Anal. Appl.
{\bf 282} (2003),
1--7.




\bibitem{KK1999}
A. G. Kusraev and S. S. Kutateladze,
{\it Boolean Valued Analysis},
Mathematics and its Applications, Springer, Netherlands, Amsterdam (1999).




\bibitem{LPV2024}
M. Lu\v{c}i\'{c}, E. Pasqualetto and J. Vojnov\'{i}c,
{\it On the reflexivity properties of Banach bundles and Banach modules},
Banach J. Math. Anal.
{\bf 18} (2024),
7.




\bibitem{MTGX2025}
X. H. Mu, Q. Tu, T. X. Guo and H. K. Xu,
{\it Common fixed point theorems for a commutative family of nonexpansive mappings in complete random normed modules},
J. Fixed Point Theory Appl.
{\bf 27} (2025),
78.




\bibitem{N1969}
S. B. Nadler, Jr.,
{\it Multivalued contraction mappings},
Pacific J. Math.
{\bf 30}(69) (1969),
475--488.




\bibitem{RV2015}
S. Rajesh and P. Veeramani,
{\it Chebyshev centers and fixed point theorems},
J. Math. Anal. Appl.
{\bf 422} (2015),
880--885.




\bibitem{S1930}
J. Schauder,
{\it Der Fixpunktsatz in Funktionalr\"{a}umen},
Studia Math.
{\bf 2}(1) (1930),
171--180.




\bibitem{SS2005}
B. Schweizer and A. Sklar,
{\it Probabilistic Metric Spaces},
Elsevier, New York (1983); Dover Publications, New York (2005).




\bibitem{SGT2025}
Y. Y. Sun, T. X. Guo and Q. Tu,
{\it A fixed point theorem for random asymptotically nonexpansive mappings},
New York J. Math.
{\bf 31} (2025),
182--194.
\bibitem{SGT2026}
Y. Y. Sun, T. X. Guo and Q. Tu,
{\it A random demiclosedness principle for random asymptotically nonexpansive mappings},
submitted to J. Nonlinear Convex Anal. 
arXiv: 2603.22764 (2026).



\bibitem{TMG2024}
Q. Tu, X. H. Mu and T. X. Guo,
{\it The random Markov-Kakutani fixed point theorem in a random locally convex module},
New York J. Math.
{\bf 30} (2024),
1196--1219.



\bibitem{TMGC2025}
Q. Tu, X. H. Mu, T. X. Guo and G. Chen,
{\it A new complete proof of the random Brouwer fixed point theorem and its implied consequences of unification},
arXiv: 2507.08521 (2025).



\bibitem{TMGYS2025}
Q. Tu, X. H. Mu, T. X. Guo, G. Yang and Y. Y. Sun,
{\it The random Kakutani fixed point theorem in random normed modules},
New York J. Math.
{\bf 31} (2025),
1543--1564.



\bibitem{W1977}
D. H. Wagner,
{\it Survey of measurable selection theorems},
SIAM J. Control Optim.
{\bf 15}(5) (1977),
859--903.



\bibitem{Xu1993}
H. K. Xu,
{\it Some random fixed point theorems for condensing and nonexpansive operators},
Proc. Amer. Math. Soc.
{\bf 10}(2) (1990),
396--400.

\bibitem{X91}
H. K. Xu,
{\it Existence and convergence for fixed points of mappings of asymptotically nonexpansive type},
Nonlinear Anal.
{\bf16} (1991),
1139--1146.
\bibitem{Z2010}
G. \v{Z}itkovi\'{c},
{\it Convex compactness and its applications},
Math. Financ. Econ.
{\bf 3}(1) (2010),
1--12.































\end{thebibliography}
\end{document}